\documentclass[reqno]{amsart}

\usepackage{graphicx} \usepackage{mathtools}
\usepackage{calrsfs} \usepackage{microtype} \usepackage{enumitem}

\usepackage{amsfonts}
\usepackage{amsmath}
\usepackage{amsrefs}
\usepackage{amssymb}
\usepackage{amstext}
\usepackage{amsthm}
\usepackage{color}
\usepackage{cancel} \usepackage{caption} \usepackage{enumitem}
\usepackage{hyperref} \usepackage{cleveref} \usepackage{subcaption} \usepackage{color}
\theoremstyle{plain}

\newtheorem{theorem}{Theorem}[section]

\newtheorem{lemma}[theorem]{Lemma}
\newtheorem{proposition}[theorem]{Proposition}

\theoremstyle{definition} 

\theoremstyle{remark} \newtheorem{remark}[theorem]{Remark}

\newtheorem{prop}[theorem]{Proposition}

 \def\cR{{\mathcal R}} 
 \def\cB{\mathcal B}
\newcommand{\cW}{\mathcal W} \newcommand{\cM}{\mathcal M}
  \def\cF{\mathcal F}
 \def\cA{\mathcal A}

\newcommand{\tr}{\mathrm{Tr}}
\newcommand{\nr}{\mathrm{N}}
\newcommand{\PG}{\mathrm{PG}}
\newcommand{\AG}{\mathrm{AG}}

 \newcommand{\car}{\text{char}}
\newcommand{\PGL}[2]{\mathrm{PGL}_{#1}(#2)}
\newcommand{\GL}[2]{\mathrm{GL}_{#1}(#2)}

\newcommand{\Aut}{\mathrm{Aut}}
\newcommand{\GF}[1]{\mathbb{F}_{#1}}

\begin{document}

\title[quasi-Hermitian varieties and applications in even
characteristic]{On quasi-Hermitian varieties in even characteristic
  and related orthogonal arrays}
\author[A. Aguglia]{Angela Aguglia}
\address{Dipartimento di Meccanica, Matematica e Management, Politecnico di Bari (IT)}
\email{angela.aguglia@poliba.it}
\author[L. Giuzzi]{Luca Giuzzi}
\address{DICATAM, Università degli Studi di Brescia (IT)}
\email{luca.giuzzi@unibs.it}
\author[A. Montinaro]{Alessandro Montinaro}
\address{Dipartimento di Matematica e Fisica "Ennio De Giorgi", Università del Salento (IT)}
\email{alessandro.montinaro@unisalento.it}
\author[V. Siconolfi]{Viola Siconolfi} \date{}
\address{Dipartimento di Meccanica, Matematica e Management, Politecnico di Bari (IT)}
\email{viola.siconolfi@poliba.it}


\begin{abstract}  
  In this paper we study the BM quasi-Hermitian varieties
  introduced in~\cite{ACK}, laying in the $3$-dimensional Desarguesian projective space of even order.
  After a brief investigation of their combinatorial properties, we
  first show that all of these varieties are projectively
  equivalent,
  exhibiting a behavior which is strikingly different
  from what happens in odd characteristic, see \cite{AL22}. This completes
  the classification project started there. Here we prove more; indeed, by using previous results, we explicitly
  determine the structure of
  the full collineation group stabilizing these varieties.
  Finally, as a byproduct of our investigation, we also construct
  a family of simple orthogonal arrays $O(q^5,q^4,q,2)$, with entries in  $\GF{q}$, where $q$ is an even prime power. Orthogonal arrays (OA's) are principally used to minimize the number of experiments
needed in order to investigate how variables in testing interact
with each other.
\end{abstract}

\maketitle

\section{Introduction}
Unitals in a finite projective plane of order $q^2$ are sets of $q^3+1$
points which have the same intersection numbers as Hermitian
curves with respect to lines, i.e. they meet every line in either $1$
or $q+1$ points. Quasi-Hermitian varieties are a natural
generalization of unitals to higher dimensions; namely they are set of
points in the finite projective space $\PG(n,q^2)$ which have the same size and the same
intersection numbers as Hermitian varieties with respect to
hyperplanes.

Actually, a point set $S$ of $\PG(n, q^2)$, $n>2$, having the same intersection numbers
with respect to hyperplanes as a non-singular Hermitian variety  has also the
same number of points as it; for $n = 2$, the size of $S$ can be either $q^3 + 1$, that
is, the size of a Hermitian curve also called a classical unital, or $q^2 + q + 1$, which is
the number of points of a Baer subplane of $\PG(2, q^2)$; see ~\cite{AA}, ~\cite{SVdV}.

It is a classical problem in finite geometry  to characterize
point-sets in term of their incidence properties with respect to
subspaces.  For instance, the notion of arc in a plane is born by
abstracting the incidence properties of a conic in a Desarguesian plane
$\PG(2,q)$.  A celebrated theorem by Segre states that for $q$ odd all
$(q+1)$-arcs are complete and turn out to be indeed conics.  As
mentioned above, the case for Hermitian curves is different, as if
$q>2$ there exist also non-classical unitals in planes of order $q^2$  (i.e. they are  not  sets of points of a
Hermitian curve).

Indeed, important families of unitals were found  by Buekenhout~\cite{B76} in every two-dimensional (projective) translation plane; Metz~\cite{M79} showed how to use Buekenhout’s method to construct a  non-classical  unital in the Desarguesian plane $\PG(2, q^2)$ for any prime power
$q > 2$. The unitals of this family are called
Buekenhout-Metz
(BM) unitals. For a careful description of these unitals
see~\cites{BE,E92} whereas for a thorough survey of the research literature on embedded
 unitals see e.g.~\cite{EB}.

As in the case of unitals, several constructions are also known for
quasi-Hermitian varieties in higher projective dimensions; see e.g.~\cites{DS,A18,ACK,LLP}.  In
particular, in~\cite{ACK} a large family of quasi-Hermitian varieties
of $\PG(n,q^2)$, depending on two parameters in the finite field $\GF{q^2}$ of order $q^2$, has been
introduced. In dimension $n=2$ these varieties are BM--unitals and thus they will be called  BM quasi-Hermitian varieties.

In~\cite{AL22}, two of the authors studied the equivalence classes, up
to projectivities, of BM quasi-Hermitian varieties for $n=3$ and $q$ odd and  they enumerated these classes, using a technique similar to the one
 employed to determine the equivalence classes number  of the
BM--unitals in the plane.

In the present paper, we consider BM quasi-Hermitian varieties in
$\PG(3,q^2)$ with $q$ even, case which was left open in~\cite{AL22},
 completing
  the classification project started there and, more explicitly, determining
 the structure of
  the full collineation group that stabilizes  these varieties.

 Precisely, in Section~\ref{cones} we explicitly recall the construction of BM
quasi-Hermitian varieties of~\cite{ACK}. In Section~\ref{s:combin} we
determine some  geometric properties of the BM quasi-Hermitian varieties $\cM_{a,b}$ of
$\PG(3,q^2)$ for $q$ even; in particular we observe that it is
possible to choose a projective reference in such a way that through
each affine point of $\cM_{a,b}$ there is exactly one line contained
in the variety, and these lines are all parallel to a given plane. It is shown
in Section~\ref{s:eqiv} that in even characteristic all of the
varieties $\cM_{a,b}$ are projectively equivalent. This is in
marked contrast with the behavior for $q$ odd. Combining this result together with some geometric features of $\cM_{a,b}$ and specific properties of suitable subgroups of $P\Gamma L_{4}(q^{2})$, the stabilizer in $P\Gamma L_{4}(q^{2})$ of the quasi-Hermitian variety $\cM_{a,b}$ is determined in Section \ref{s:stab}. Further, its structure and its action on the
points of $\cM_{a,b}$ is discussed.

Our long-term
aim is to try to find a characterization of the BM quasi-Hermitian
varieties among all possible quasi-Hermitian varieties in spaces of
the same dimension and order.

Finally, in
Section~\ref{oarr} simple orthogonal arrays $OA(q^5,q^4,q,2)$ of index
$q^3$, $q$ even, are constructed from  the BM quasi-Hermitian varieties $\cM_{a,b}$.

Orthogonal arrays (OA's) are principally used to minimize the number of experiments
needed in order to investigate how variables in testing interact
with each other and, consequently, determine the required parameters.
For instance, OA's are used to calibrate the flight parameters of drones,
in order to optimize their performance; see e.g. \cite{UAV}.

\medskip

\section{Background on quasi-Hermitian varieties}
\label{cones}

Quasi-Hermitian varieties were introduced in~\cite{DS} as a
generalization of non-singular Hermitian varieties through the
following definition.

  A point-set $H$ in $\PG(n,q^2)$ is a \emph{quasi-Hermitian variety}
  if has the same size and the same intersection numbers with
  hyperplanes as a non-singular Hermitian variety $H(n,q^2)$ of
  $\PG(n,q^2)$.

In particular, a quasi-Hermitian variety is a set of size
$(q^{n+1}+(-1)^n)(q^{n}-(-1)^n)/(q^2-1)$ of $\PG(n,q^2)$ meeting the
hyperplanes in either
\[ (q^{n}+(-1)^{n-1})(q^{n-1}-(-1)^{n-1})/(q^2-1)\] or
\[ 1+q^2(q^{n-1}+(-1)^n)(q^{n-2}-(-1)^n)/(q^2-1) \] points;
see~\cite{S}.

\begin{remark}
    As pointed out in the introduction, it is possible to drop the
  requirement on the size of $H$ if $n\geq 3$.
\end{remark}

There are a few families known of quasi-Hermitian varieties; some of them
turn out to be a sort of higher-dimension analogue
to the known families of unitals or can be obtained from them
by pivoting; see~\cite{SVdV}. Here we point out that there are also families
for $n\geq 3$ which are quite different from those; see~\cite{LS}.

The quasi-Hermitian varieties we are considering in the present
paper are  BM quasi-Hermitian varieties as for $n=2$ they turn out to be
Buekenhout-Metz unitals; see e.g. \cite{EB}.
They are defined as follows.

Let $\cB_{a,b}$ be the surface of $\PG(3,q^2)$ of projective equation
\begin{multline}
  \label{eq:bab}
  \cB_{a,b}: Z^qJ^q-ZJ^{2q-1}+a^q(X^{2q}+Y^{2q})-a(X^2+Y^2)J^{2q-2}=\\
  (b^q-b)(X^{q+1}+Y^{q+1})J^{q-1},
\end{multline}
with $a\in\GF{q^2}^*$ and $b \in\GF{q^2}\setminus\GF{q}$.  Denote by
$\Sigma_{\infty}$ the hyperplane at
infinity with equation $J=0$ of $\PG(3,q^2)$, put
\begin{equation}
 \label{cono}
\mathcal{F}:=\{(0,X,Y,Z)\in \PG(3,q^2)| X^{q+1}+Y^{q+1}=0\}
\end{equation}
and
\[
  \cB_{\infty}:=(\cB_{a,b}\cap \Sigma_{\infty}).
\]

In \cite{ACK} it was proved that the following point set of $\PG(3,q^2)$

\begin{equation}\label{mab}
  \cM_{a,b}:=(\cB_{a,b}\setminus\cB_{\infty})\cup \mathcal{F}
\end{equation}
is a quasi-Hermitian variety for $q\geq 4$ even or for $q$ odd and $4a^{q+1}+(b^q-b)^2\neq 0$. We call it
a {\em BM quasi-Hermitian variety}.

Clearly, the affine points of $\cM_{a,b}$ satisfy the affine equation:
\begin{multline}
  \label{eq:mab}
  \cB_{a,b}: Z^q-Z+a^q(X^{2q}+Y^{2q})-a(X^2+Y^2)=
  (b^q-b)(X^{q+1}+Y^{q+1}).
\end{multline}
If $q=2$ then $b+b^q=1$, and hence $\cM_{a,b}$ is an Hermitian variety with equation 
\begin{multline}
  \label{eq:q=2}
  Z^qJ+ZJ^q+a^q(X+Y)J^q+a(X^q+Y^q)J+(X^{q+1}+Y^{q+1})=0.
\end{multline}
Therefore, in the sequel we assume that $q>2$.

Finally we recall that any
two BM quasi-Hermitian varieties  $\cM_{a,b}$  and $\cM_{\alpha,\beta}$ of $\PG(3,q^2)$ are \emph{projectively equivalent} if there exists a  collineation $\psi \in
P \Gamma L_4({q^2})$ such that $\psi(\cM_{a,b})=\cM_{\alpha,\beta}$.



\section{Preliminaries on BM quasi-Hermitian varieties in $\PG(3,q^2)$, $q$ even}
\label{s:combin}
Let $q>2$ be an even prime power, $a\in\GF{q^2}^*$ and $b \in\GF{q^2}\setminus\GF{q}$.
 In the present section we first
determine  the number of lines through a point of $\cB_{a,b}$ with equation \eqref{eq:bab} which are contained in $\cB_{a,b}$, next we deduce some related combinatorial properties of the BM quasi-Hermitian  variety $\cM_{a,b}$ (see \eqref{mab}) and some information about the stabilizer  of $\cM_{a,b}$ in the projective linear group $\PGL{4}{q^2}$.

\begin{theorem} \label{th32} Let $\cB_{a,b}$ be the surface of
  equation~\eqref{eq:bab} in $\PG(3,q^2)$, $q$ an even prime power and put $\cB_{\infty}=\cB_{a,b}\cap [J=0]$ and $P_{\infty}=(0,0,0,1)$.
  Then,
  \begin{itemize}
  \item[(i)] for any affine point $Q$ of $\cB_{a,b}$ there is exactly one
    line of $\PG(3,q^2)$ passing through $Q$ and contained in
    $\cB_{a,b}$;
  \item[(ii)] for any point $R$ in $\cB_{\infty}\setminus P_{\infty}$ there
    are $q+1$ lines contained in $\cB_{a,b}$ passing  through $R$ (and exactly
    one of these lines is contained in $\cB_{\infty}$);
  \item[(iii)] there is exactly one line, among the ones contained in
    $\cB_{a,b}$, that passes through $P_{\infty}$ and it consists of the points of
    $\cB_{\infty}$.
  \end{itemize}

\end{theorem}

\begin{proof}
  Observe that for $q$ even
  \[
    \cB_{\infty}:
    \begin{cases}
      J=0 \\
      (X+Y)^{2q}=0.
    \end{cases}
  \]
  This is the line in $\Sigma_{\infty}:=[J=0]$ of equations $X+Y=0=J$.
  To stress this fact we shall call it $\ell_{\infty}$. We refer to
  the points in $\ell_{\infty}$ as $M_{\infty}=(0,1,1,0)$,
  $P_{\infty}=(0,0,0,1)$ and $L^m_{\infty}=(0,m,m,1)$ with
  $m\in\GF{q^2}^*$.

Any line $\ell\in \cB_{a.b}$ which is not contained in $\Sigma_{\infty}$,  	
  must contain one of the points in $\ell_{\infty}=\cB_{\infty}$.
	
  Take $P\in \cB_{a,b}\cap\AG(3,q^2)$, it is known~\cite{ACK} that the
  collineation group of $\cB_{a,b}$ acts transitively on the points of
  $\cB_{a,b}\cap\AG(3,q^2)$, thus we can assume without loss of
  generality $O=(0,0,0)$ and that $\ell$ has the following affine
  parametric equations:
  \[
    \begin{cases}
      x=m_1t \\
      y=m_2t \\
      z=m_3t
    \end{cases}
  \]
  for $t\in\GF{q^2}$ and
  $(m_1,m_2,m_3)\in \{(0,0,1),(m,m,1),(1,1,0)\}$. One can easily
  notice that:
  \begin{itemize}
  \item $(m_1,m_2,m_3)\neq (0,0,1)$, for otherwise the line would not be
    contained in $\cB_{a,b}$;
  \item $(m_1,m_2,m_3)\neq (m,m,1)$ because $\car(\mathbb{K})=2$ and
    again $\ell$ would not be contained in $\cB_{a,b}$.
  \end{itemize}
	
  So, we conclude that the only possible line contained in $\cB_{a,b}$
  and passing through $P_{\infty}$ has affine representation
  \[
    \begin{cases}
      x=t \\
      y=t \\
      z=0,
    \end{cases}
    t\in\GF{q^2}.
  \]
  Inspection of equation~\eqref{eq:bab} shows that this line is
  actually contained in $\cB_{a,b}$.  Using now the transitivity of
  the collineation group on the affine points of $\cB_{a,b}$ we obtain
  that for any point in $\cB_{a,b}\cap\AG(3,q^2)$ passes one and only
  one line contained in $\cB_{a,b}$.
	
  Now we turn our attention at the points of $\ell_{\infty}=\cB_{\infty}$, in
  particular we count the lines in $\cB_{a,b}$ that contain
  $L^{m}_{\infty}=(0,m,m,1)$ and are not $\ell_{\infty}$. The general
  line $r$ with this property has affine parametric equations
  \[
    r:\begin{cases}
        x=\bar{x}+mt \\
        y=\bar{y}+mt \\
        z=\bar{z}+t,
      \end{cases}
    \]
    where $(\bar{x},\bar{y},\bar{z})\in \cB_{a,b}\cap\AG(3,q^2)$,
    which means that:
    \begin{equation}\label{eq:x_in_bab}
      \bar{z}^q+\bar{z}+a^q(\bar{x}^{2q}+\bar{y}^{2q})+a(\bar{x}^2+\bar{y}^2)=(b^q+b)(\bar{x}^{q+1}+\bar{y}^{q+1}).
    \end{equation}
    We now write the condition for the whole line $r$ to be contained
    in $\cB_{a,b}$:
    \begin{align*}
      & \bar{z}^q+t^q+\bar{z}+t+a^q(\bar{x}^{2q}+\bar{y}^{2q}+\underbrace{(mt)^{2q}+(mt)^{2q}}_{=0})
        +a(\bar{x}+\bar{y})^2+\underbrace{(mt)^{2}+(mt)^{2}}_{=0})=                                       \\
      = & (b^q-b)[(\bar{x}^q+m^qt^q)(\bar{x}+mt)+(\bar{y}^q+m^qt^q)(\bar{y}+mt)].
    \end{align*}
    Simplifying~\eqref{eq:x_in_bab} we obtain
    \begin{align*}
      t^q[(m^q(b^q+b)(\bar{x}+\bar{y}))+1]+t[m(b^q+b)(\bar{x}+\bar{y})^q+1]=0.
    \end{align*}
    In order to have the latter equation satisfied for any
    $t\in\GF{q^2}$, we must have
    \begin{equation}
      (\bar{x}+\bar{y})^q=\frac{1}{m(b^q+b)} \text{ equivalently, }\; (\bar{x}+\bar{y})=\frac{1}{m^q(b^q+b)}.
    \end{equation}
    Given any $m$, there are $q^{2}$ possible pairs
    $(\bar{x},\bar{y})$ that satisfy
    $(\bar{x}+\bar{y})^q=\frac{1}{m(b^q+b)}$. For any such pair
    $(\bar{x},\bar{y})$, there are $q$ possible values of $\bar{z}$
    that satisfy $(\bar{x},\bar{y},\bar{z})\in \cB_{a,b}$. We deduce
    that the number of lines passing through $L^{m}_{\infty}$
    contained in $\cB_{a,b}$ is $\frac{q^2q}{q^2}+1=q+1$.
	
    We can repeat the same argument for $M_{\infty}=(0,1,1,0)$ and
    count the lines in $\cB_{a,b}$ through $M_{\infty}$. Consider the
    general affine line $r$ such that $M_{\infty}\in r$ and
    $r\neq\ell_{\infty}$:
    \[
      r:\begin{cases}
          x=\bar{x}+t \\
          y=\bar{y}+t \\
          z=\bar{z},
        \end{cases}
      \]
      with $(\bar{x},\bar{y},\bar{z})\in
      \cB_{a,b}\cap\AG(3,q^2)$. Reasoning as for $L^{m}_{\infty}$ we
      obtain
      \[
        t^q[((b^q+b)(\bar{x}+\bar{y}))]+t[(b^q+b)(\bar{x}+\bar{y})^q]=0.
      \]
      This equality is satisfied for every $t\in\GF{q^2}$ if and only
      if $\bar{x}=\bar{y}$. Notice that for every $\bar{x}=\bar{y}$
      there are $q$ possible $\bar{z}$ such that
      $(\bar{x},\bar{y},\bar{z})\in \cB_{a,b}\cap \AG(3,q^2)$. So, we
      obtain $q$ possible lines passing through $M_{\infty}$ with
      $r\neq\ell_{\infty}$.
	
     The general line passing
      through $P_{\infty}$ and not entirely contained in
      $\Sigma_{\infty}$ has affine equation
      \[
        r:\begin{cases}
            x=\bar{x} \\
            y=\bar{y} \\
            z=\bar{z}+t.
          \end{cases}
	\]
	We require that
        $(\bar{x},\bar{y},\bar{z})\in \cB_{a,b}\cap\AG(3,q^2)$ and
        $r\subset \cB_{a,b}$. This implies $t^q+t=0$ for any
        $t\in\GF{q^2}$, which is not true. We conclude that the only
        line contained in $\cB_{a,b}$ passing through $P_{\infty}$ is
        $\ell_{\infty}$.
      \end{proof}

\begin{remark}\label{rmab}
  Observe that for every point $L^{m}_{\infty}$ or $M_{\infty}$ the $q$
  affine lines in $\cB_{a,b}$ containing it are coplanar. In
  particular, the general affine line of $\cB_{a,b}$ through $M_{\infty}$ is
  contained in the plane of equation $x+y=0$, while the general affine
  line of $\cB_{a,b}$ through $L^m_{\infty}$ is contained in the affine plane
  $x+y=\frac{1}{m^q(b^q+b)}$.
  Furthermore $\ell_{\infty}\subset \mathcal{F}$, where $\mathcal{F}$ is the Hermitian cone defined in \eqref{cono}.
\end{remark}

\bigskip
\begin{theorem} \label{prop:lines_bab}

  Let $\cM_{a,b}$ be the BM quasi-Hermitian variety of $\PG(3,q^2)$, $q$ even,  described by~\eqref{mab}.
  Then through each affine point of $\cM_{a,b}$ there passes one line of $\cM_{a,b}$, whereas through a point at infinity of $\cM_{a,b}\cap \ell_{\infty}$ there pass $q+1$ lines of a pencil contained in $\cM_{a,b}$;
\end{theorem}
\begin{proof}
We observe that the affine points of $\cM_{a,b}$ are the same as those of $\cB_{a,b}$, whereas the set $\cF$ of points at infinity of $\cM_{a,b}$ consists of the points $P=(0,x,y,z)$ such that  $x^{q+1}+y^{q+1}=0$ and it contains the points at infinity of $\cB_{\infty}=\ell_{\infty}$.
Hence, from Theorem \ref{th32}  and Remark \ref{rmab} we get the result.
\end{proof}

Now, denote by $G$
the stabilizer of $\cM_{a,b}$ in the projective linear group
$\PGL{4}{q^2}$.

\begin{lemma}
  \label{autaut0}
  The group $G$ stabilizes the affine points of $\cM_{a,b}$, fixes the
  point $P_{\infty}$ and preserves both the line $\ell_{\infty}$ and the
  hyperplane $\Sigma_{\infty}$.
\end{lemma}
\begin{proof}
  By Theorem~\ref{prop:lines_bab}, the points of $\ell_{\infty}$ are
  the only points of $\cM_{a,b}$ through which more than one line of
  $\cM_{a,b}$ passes. So, any element of $G$ must map a point of
  $\ell_{\infty}$ onto a point of $\ell_{\infty}$.  We also know
	by \cite{ACK}*{Corollary 4.3} that $G$ acts transitively on the
  affine points of $\cM_{a,b}$. In particular, since for any affine
  point $Q$ of $\cM_{a,b}$ there is exactly one line $\ell_Q$ meeting
  $\ell_{\infty}$ in a point different from $P_{\infty}$, we get that
  $G$ is also transitive on the points of
  $\ell_{\infty}\setminus \{P_{\infty}\}$ and fixes $P_{\infty}$ itself.
  Finally, as $P_{\infty}$ is fixed by $G$, any collineation in $G$
  must send lines through $P_{\infty}$ to lines through $P_{\infty}$.
  However, all lines though $P_{\infty}$ in $\cM_{a,b}$ are contained
  in $\Sigma_{\infty}$ (and they actually span this hyperplane).  So
  $G$ stabilizes $\Sigma_{\infty}$ too.
\end{proof}

\begin{remark}
 \label{remNot}
	By Lemma \ref{autaut0}, the group $G$ is an affine group of collineations, as
	it fixes the hyperplane at infinity. As such, we can represent the elements of $G$
	by  $4\times 4$ matrices with elements in $ \GF{q^{2}}$ of the form
	\[M=\begin{pmatrix}
		1 & \gamma_{1} & \gamma_{2} & \gamma_{3} \\
		0 & d&e&h         \\
		0 &  f&g&i       \\
		0 &0&0& c
	\end{pmatrix},
	\]
	where $c(dg+ef)\neq 0$ and  $d+f=e+g$.
\end{remark}
The first column of $M$ is $(1,0,0,0)^t$ because $\phi(\Sigma_{\infty})=\Sigma_{\infty}$;
the last row of $M$ is $(0,0,0,c)$ because $\phi(P_{\infty})=P_{\infty}$. Furthermore $d+f=e+g$ since $\phi$ preserves the line $\ell_{\infty }$.

\section{Projective equivalence of $\cM_{a,b}$'s}
\label{s:eqiv}

In this section we are going to prove that the BM quasi-Hermitian varieties in $\PG(3,q^2)$, $q>2$ even are equivalent.
Let $\phi$ in $P\Gamma L_4(q^2)$. We
represent $\phi$ by a  non--singular
matrix $M$ together with a field automorphism $\sigma$. By convention, to
apply $\phi$ to some point we first apply $\sigma$ to each entry of the row vector
representing the point and then multiply on the right by the matrix $M$.
The maps $\tr:\GF{q^2}\to\GF{q}, x \mapsto x+x^{q}$ and $N:\GF{q^2}\to\GF{q}, x \mapsto x^{q+1}$ are the $\GF{q}$-trace and the $\GF{q}$-norm, respectively.

\begin{lemma}\label{fin}
	$\cM_{a,b}$ and $\cM_{a',b'}$ are equivalent as quasi-Hermitian varieties if and only if there is a collineation $\phi: \cM_{a,b}\mapsto\cM_{a',b'}$
	with associated field automorphism $\sigma$ and represented
        by a matrix
	\[
		M=\begin{pmatrix}
			1 & 0           & 0           & 0  \\
			0 & d           & e           & 0  \\
			0 & \lambda_1 e & \lambda_2 d & 0  \\
			0 & 0           & 0           & c
		\end{pmatrix}
	\]
	where $c\in\GF{q}^*$, $d,e\in\GF{q^2}$,  $\lambda_1,\lambda_2\in\GF{q^2}$ such that $\lambda_1^{q+1}=\lambda_2^{q+1}=1$  and one of the following holds
	\begin{enumerate}[label={\rm (\Roman*)}]
        \item\label{I}	 $e=0$, $d\neq 0$ and $\lambda_2=1$;
        \item\label{II}	 $e\neq0$,  $d=0$ and $\lambda_1=1$;
        \item\label{III}  $e\neq0\neq d$, $\lambda_1=\lambda_2=1$, and $d/e \in \GF{q}\setminus\{1\}$;	
        \item\label{IV}  $e\neq0\neq d$, $\lambda_2 \neq 1 \neq \lambda_1$, $\lambda_1\neq \lambda_2$ and  $d=\frac{(1+\lambda_1)}{(1+\lambda_2)}e$.
	\end{enumerate}
\end{lemma}

\begin{proof}
	It is enough to prove the necessary condition. So assume that there is $\phi \in P\Gamma{}L_4(q^2)$ such that $\phi(\cM_{a,b})=\cM_{a',b'}$.
	Let $G_{1}=\Aut(\cM_{a,b})$ and $G_{2}=\Aut(\cM_{a',b'})$, then $G_{1}^{\phi}=G_{2}$.
	Since $(P_{\infty },\ell_{\infty },\Sigma_{\infty })$ are the unique
	subspaces of $PG_{3}(q^{2})$, preserved by $G_{i}$, $i=1,2$, having
	non-empty intersection with $\cM_{a,b}$ and $\cM_{a',b'}$, then $\phi $ preserves $%
		(P_{\infty },\ell_{\infty },\Sigma_{\infty })$. Further, $\phi $ preserves
	the Hermitian cone $\mathcal{C=}\cM_{a,b}\cap \Sigma_{\infty }=\cM_{a',b'}\cap
		\Sigma_{\infty }$, hence $\phi $ preserves each of the subsets $P_{\infty
			},\ell_{\infty },\mathcal{F},\Sigma_{\infty }$.
	
	Since both $G_{1}$ and $G_{2}$ act transitively on $\cM_{a,b}\setminus \Sigma
		_{\infty }$ and on $\cM_{a',b'}\setminus \Sigma_{\infty }$ by \cite[Corollary 4.3]{ACK}, there is $g_{i}\in G_{i}$
	such that $g_{2}\phi g_{1}$ is an isomorphism from $\cM_{a,b}$ onto $\cM_{a',b'}$
	fixing $O$. Thus, w.l.o.g. $\phi $ fixes $O=(1,0,0,0)$, and hence from Remark \ref{remNot} it is represented by the non-singular matrix
	
	\begin{equation}
		\label{s:phi}
		M=\begin{pmatrix}
			1 & 0 & 0 & 0  \\
			0 & d & e & h  \\
			0 & f & g & i  \\
			0 & 0 & 0 & c
		\end{pmatrix}.
	\end{equation}
	where $d+f=e+g$.	
	For any $\alpha\in\GF{q}$, $(1,0,0,\alpha)\in \cM_{a,b}$. So we deduce from~\eqref{s:phi} that necessarily
	$\phi(1,0,0,\alpha)=(1,0,0,c\alpha)\in \cM_{a',b'}$
	and hence $c\in\GF{q}^{*}$.
 
We consider now the plane of equation $Y=0$. Its intersection with $\cM_{a,b}$
	is given by a set of points
	$(1,x,0,z)$ such that
	\[
		ax^2+bx^{q+1}+z\in\GF{q},
	\]
	which implies that
	\begin{equation}\label{eq:1}
		a^{\sigma}{x^{2\sigma}}+b^{\sigma}{x^{\sigma(q+1)}}+z^{\sigma}\in\GF{q}.
	\end{equation}
	Suppose $\phi(1,x,0,z)\in\cM_{a',b'}$; then
	\[
		a'\frac{(d^2+e^2)}{c}x^{2\sigma}+b'\frac{(d^{q+1}+e^{q+1})}{c}x^{\sigma(q+1)}+\frac{h}{c}x^{\sigma}+z^{\sigma}\in\GF{q};
	\]
	this and~\eqref{eq:1} together give:
	\begin{equation}\label{eq:2}
		(a^{\sigma}+c^{-1}a'(d^2+e^2))x^{2\sigma}+(b^{\sigma}+c^{-1}b'(d^{q+1}+e^{q+1}))x^{\sigma(q+1)}+\frac{h}{c}x^{\sigma} \in\GF{q}.
	\end{equation}
 
	As $q>2$, we can choose  a primitive element  $\delta $ of $\GF{q^2}$ such that
	$\delta^q=1+\delta$, $\delta^2+\delta+\lambda=0$,  $\delta^{q+1}=\lambda\neq 0$ and the absolute trace of $\lambda$ equals $1$.

	We substitute in~\eqref{eq:2} the following values of $x^{\sigma}$ :
	\begin{enumerate}[label={\roman*})]
		\item\label{i-x}
		$x^{\sigma}=1$; thus
		\[
			(a^{\sigma}+c^{-1}a'(d^2+e^2))+(b^{\sigma}+c^{-1}b'(d^{q+1}+e^{q+1}))+\frac{h}{c}
			\in\GF{q};
		\]
		\item\label{ii-x}
		$x^{\sigma}=\delta$; thus
		\[
			(a^{\sigma}+c^{-1}a'(d^2+e^2))\delta^2+(b^{\sigma}+c^{-1}b'(d^{q+1}+e^{q+1}))\delta^{q+1}+\frac{h}{c}\delta
			\in\GF{q};
		\]
		
		\item\label{iii-x} $x^{\sigma}=\delta^q$ thus
		\[
			(a^{\sigma}+c^{-1}a'(d^2+e^2))\delta^{2q}+(b^{\sigma}+c^{-1}b'(d^{q+1}+e^{q+1}))\delta^{q+1}+\frac{h}{c}\delta^q
			\in\GF{q};
		\]
		\item\label{iv-x} $x^{\sigma}=\lambda$ thus
		\[
			(a^{\sigma}+c^{-1}a'(d^2+e^2))\lambda^2+(b^{\sigma}+c^{-1}b'(d^{q+1}+e^{q+1}))\lambda^2+\frac{h}{c}\lambda \in\GF{q};
		\]
		\item\label{v-x} $x^{\sigma}=\lambda\delta$ thus
		\[
			(a^{\sigma}+c^{-1}a'(d^2+e^2))(\delta\lambda^2)+(b^{\sigma}+c^{-1}b'(d^{q+1}+e^{q+1}))\lambda^2\delta^{q+1}+\frac{h}{c}\lambda\delta \in\GF{q}.
		\]
	\end{enumerate}
	Let $A=a^{\sigma}+a'\frac{(d^2+e^2)}{c}$; $B=(b^{\sigma}+b'\frac{(d^{q+1}+e^{q+1})}{c})$; $C=\frac{h}{c}$.
	
	From \ref{i-x} we obtain that $A=B+C+s$ where $s\in\GF{q}$, substituting it in \ref{iv-x} we obtain
	\begin{multline*}
		\cancel{B\lambda^2}+{C\lambda^2}+s\lambda^2+\cancel{B\lambda^2}+{C\lambda}\in\GF{q}\Rightarrow C\lambda(\lambda+1)\in\GF{q}\Rightarrow \\ C(\lambda+1)\in\GF{q} \Rightarrow C\in\GF{q}.
	\end{multline*}
	Summing up \ref{ii-x} and \ref{iii-x} we obtain
	\[
		A\underbrace{(\delta^2+\delta^{2q})}_{=1}+C\underbrace{(\delta+\delta^q)}_{=1}\in\GF{q}\Rightarrow A+C\in\GF{q}\Rightarrow A,C\in\GF{q}.
	\]
	Furthermore from \ref{ii-x} and $\delta^2=\lambda+\delta$
	\[
		A(\delta+\lambda)+B\lambda+C\delta\in\GF{q}\Rightarrow A\delta+C\delta \in\GF{q} \underbrace{\Rightarrow}_{\delta \notin\GF{q}} A=C.
	\]
	Finally,
	from \ref{v-x}
	\[
		A\lambda\delta+c\delta \in\GF{q} \Rightarrow A\lambda+C=0 , A=C=0.
	\]
	It follows that
	
	\begin{equation}\label{fin1}
		a^{\sigma}+a'\frac{(d^2+e^2)}{c}=0;
	\end{equation}

	\begin{equation}\label{fin2}
		(b^{\sigma}+b'\frac{(d^{q+1}+e^{q+1})}{c})\in\GF{q};
	\end{equation}
	and
	\[
		h=0.
	\]
	With a very similar argument with respect to the plane of equation $X=0$
	we can conclude
	\[
		a^{\sigma}+a'\frac{(f^2+g^2)}{c}=0;\text{ }(b^{\sigma}+b'\frac{(f^{q+1}+g^{q+1})}{c})\in\GF{q};\text{ }i=0.
	\]
	So,
	\begin{equation}\label{eq:3}
		d^2+e^2=f^2+g^2\neq 0 \text{ and } f^{q+1}+g^{q+1}=d^{q+1}+e^{q+1}\neq 0.
	\end{equation}
	
	We now know that $(1,x,y,z)\in \cM_{a,b}$ if and only if
	$\phi(1,x,y,z)\in \cM_{a',b'}$ and also
	$(1,x,y,z)\in \cM_{a,b}$ if and only if
	\begin{equation}
		a(x^2+y^2)+b(x^{q+1}+y^{q+1})+z\in\GF{q}.
	\end{equation}
	On the other hand, $\phi(1,x,y,z)=(1,dx^{\sigma}+fy^{\sigma},ex^{\sigma}+gy^{\sigma},cz^{\sigma})$, so
	$\phi(1,x,y,z)\in\cM_{a',b'}$ if and only if
	\[
		c^{-1}( a'(dx^{\sigma}+fy^{\sigma})^2+(ex^{\sigma}+gy^{\sigma})^2)
	\]
	\[
		+c^{-1}b'((dx^{\sigma}+fy^{\sigma})^{q+1}+(ex^{\sigma}+gy^{\sigma})^{q+1})+z^{\sigma}\in\GF{q}.
	\]
	This together with $(1,x,y,z)\in \cM_{a,b}$ leads to
	\[
		a^{\sigma}(x^{2\sigma}+y^{2\sigma})+ a'(\frac{(dx^{\sigma}+fy^{\sigma})^2}{c}+\frac{(ex^{\sigma}+gy^{\sigma})^2}{c})
	\]
	\[
		b^{\sigma}(x^{\sigma(q+1)}+y^{\sigma(q+1)})+b'\frac{(dx^{\sigma}+fy^{\sigma})^{q+1}}{c}+\frac{(ex^{\sigma}+gy^{\sigma})^{q+1})}{c} \in\GF{q}.
	\]
	Using~ \eqref{fin1}, \eqref{fin2} and \eqref{eq:3} we obtain
	\begin{equation}\label{eq:4}
		b'[(d^qf+e^qg)x^{\sigma q}y^{\sigma}+(df^q+eg^q)x^{\sigma}y^{\sigma q}]\in\GF{q}.
	\end{equation}
	
	Let $\omega \in\GF{q^2}$ be a solution of $\xi^{q+1}=1$. Since $\phi$ has to leave invariant the Hermitian cone $\cM_{a,b}\cap \Sigma_{\infty}$, we have $\phi(0,x,\omega x,z)\in \cM_{a,b}\cap \Sigma_{\infty}$. Again using~\eqref{eq:3} we have:
	\[
		(d^qf+e^qg)\omega^{\sigma}+(df^q+eg^q)\omega^{\sigma q}=0
	\]
	for any of the $q+1$ values $\omega$ such that $\omega^{q+1}=1$. This means that we found $q+1$ solutions to an equation of degree $q$, so it must be
	\begin{equation}\label{eq:5}
		d^qf+e^qg=0
	\end{equation}
	If $e=0$, since $d\neq 0$, we get $f=0$ and from $e+d=f+g$ we obtain $d=g$ that is~\ref{I}. Suppose $e \neq 0$. If $d=0$ then also $g=0$ and $e=f$, that is~\ref{II}.
	 If $d\neq 0$, as also $f\neq 0$  from \eqref{eq:5}
	we have$ (\frac{d}{e})^q=\frac{g}{f}$.

	From~\eqref{eq:3} we have:
	\begin{align*}
		(d+e)^{q+1}               & =(f+g)^{q+1}                  \\
		d^{q+1}+d^qe+e^qd+e^{q+1} & =f^{q+1}+f^qg+g^qf+g^{q+1}    \\
		d^qe+e^qd                 & =f^qg+g^qf=\lambda\in\GF{q}.
	\end{align*}

	Thus, we get
	\begin{align*}
		\frac{d^{q}}{e^q} & =\frac{\lambda}{e^{q+1}}+\frac{d}{e} \\
		\frac{g^{q}}{f^q} & =\frac{\lambda}{f^{q+1}}+\frac{g}{f}
	\end{align*}
	and hence $\frac{d}{e}+\frac{g}{f}=\frac{\lambda}{e^{q+1}}=\frac{\lambda}{f^{q+1}}$,
	which implies $\lambda=0$ or $e^{q+1}=f^{q+1}$ and $d^{q+1}=g^{q+1}$.
	
	In the case in which $\lambda=0$ then $df+ge=0$. Put $d/e=g/f=\alpha$. Then $\alpha \neq 1$ and from $d+e=g+f$ we get
	$e(\alpha+1)=(\alpha+1)f$. Hence $e=f$ and $d=g$, that is case~\ref{III} holds.
	
	If $\lambda \neq 0$ then $f=\lambda_1 e$ and $g=\lambda_2 d$ such that $\nr(\lambda_1)=\nr(\lambda_2)=1$. If $\lambda_1=\lambda_2=1$, then we get again
	case~\ref{III}.
	If $\lambda_1\neq 1$, then also $\lambda_2\neq 1$ and we get $(1+\lambda_2)d=(1+\lambda_1)e$, furthermore if $\lambda_1=\lambda_2$ then $d=e$ and $M$ would be singular so $\lambda_1\neq \lambda_2$, that is case~\ref{IV}
	.
	This concludes the proof.
\end{proof}

From the previous Lemma, taking into account  conditions  \eqref{fin1} and \eqref{fin2}, we get the following.

\begin{lemma}
	\label{nolinear}Let $(a,b), (a',b') \in \GF{q^2}^*\times (\GF{q^2}\setminus \GF{q})$ with  $(a',b')\neq (a,b)$. There is $\phi \in
		P\Gamma L_{4}(q^{2})$ such that $\cM_{a,b}^{\phi }=\cM_{a',b'}$ if and only if
	\[
		\left\{
		\begin{array}{ccl}
			a' & = & ca^{\sigma}/(d^2+e^2)            \\
			b' & = & cb^{\sigma}/(d^{q+1}+e^{q+1})+u
		\end{array}%
		\right.
	\]
	for some $c \in\GF{q}^{\ast}$ $u \in \GF{q}$, and $d,e$ satisfying the conditions of Lemma~\ref{fin}.
\end{lemma}

Assume that $\cM_{a,b}$ and $\cM_{a',b'}$ are projectively equivalent. In this case we write $(a,b) \sim (a',b')$ where $\sim$ is in particular an equivalence relation on the ordered pairs $(a,b)\in \GF{q^2}^2$
such that $a \neq 0$ and $b \in \GF{q^2}\setminus \GF{q}$.

\begin{lemma}\label{lemadd5}
	Let $\cM_{a,b}$ be a BM quasi-Hermitian variety of $\PG(3,q^2)$, $q$ even
	and $\varepsilon$ be a primitive element of $\GF{q^2}$.
	Then, there exists $\alpha\in\GF{q^2}^*$ such that
	$\cM_{a,b}$ is  equivalent to $\cM_{\alpha,\varepsilon}$.
\end{lemma}
\begin{proof}
  Write $b=b_0+\varepsilon b_1$, with $b_0,b_1 \in \GF{q}$ and $b_1 \neq 0$. Then, there exists $d\in \GF{q^2}^*$, such that  $b_1/d^{q+1}=1$.
  Therefore one can apply \Cref{nolinear} with $ c=1, e=0, u=b_0/d^{q+1}$ and obtain $(a,b)
 \sim (a/d^2, b/d^{q+1}+b_0/d^{q+1})=(a/d^2, \varepsilon) $.
\end{proof}

\begin{theorem}
\label{mainequiv}
	All BM quasi-Hermitian varieties of $\PG(3,q^2)$, $q$ even, are equivalent.
\end{theorem}
\begin{proof}
  In light of Lemma~\ref{lemadd5},
  in order to determine the equivalence
  classes of BM quasi-Hermitian varieties it is enough to determine when
  two varieties $\cM_{a,\varepsilon}$ and $\cM_{a',\varepsilon}$
  are linearly equivalent.
  In particular, we consider the case $\sigma=id$.
  Then, $\cM_{a,\varepsilon}$ and $\cM_{a',\varepsilon}$ are equivalent
  if and only if
  \[ a'=ca/(d^2+e^2); \]
  \[ \varepsilon(1+c/(d^{q+1}+e^{q+1}))=u. \]
  As $1+c/(d^{q+1}+e^{q+1})\in\GF{q}$ and $u\in\GF{q}$, we must have
  $c/(d^{q+1}+e^{q+1})=1$ for the second equation to be possible.
  Replacing in the first equation we get
  \[ a'=a\frac{d^{q+1}+e^{q+1}}{d^2+e^2}. \]
  We claim that this yields just one equivalence class; this is the same
  as to say that the function
  \[ (d,e)\mapsto\frac{d^{q+1}+e^{q+1}}{d^2+e^2} \]
  is surjective on $\GF{q^2}^*$.

 We know that $d(1+\lambda_2)=e(1+\lambda_1)$ with $\nr(\lambda_1)=\nr(\lambda_2)=1$.  Assume $\lambda_2\neq 1$ and put $\beta=\frac{1+\lambda_1}{1+\lambda_2}$
 Hence we have to prove that  for each $m^2\in\GF{q^2}$ (recall that
 in characteristic $2$ the map $x\mapsto x^2$ is bijective) there are  $e, \lambda_1,\lambda_2 \in \GF{q^2} $ such that

\[e^{q-1}=m^2 \frac{(1+\beta)^2}{(1+\beta^{q+1})}.\]
 This is possible if and only if
\[m^{2(q+1)} \frac{(1+\beta)^{2(q+1)}}{(1+\beta^{q+1})^2}=1\]
that is
\[m^{q+1}\frac{(1+\beta)^{(q+1)}}{(1+\beta^{q+1})}=1,\]
whence
\[m^{q+1}(\lambda_1^q+\lambda_1)=\lambda_2^q+\lambda_2+\lambda_1^q+\lambda_1.\]
For a chosen $\lambda_1$ we have to find $\lambda_2$ such that
\[  \lambda_2^q+\lambda_2=(1+m^{q+1})(\lambda_1^q+\lambda_1 ), \]
 that is
 \[  \lambda_2^2+(1+m^{q+1})(\lambda_1^q+\lambda_1 )\lambda_2+1=0. \]
Since the absolute trace of $\frac{1}{(1+m^{q+1})(\lambda_1^q+\lambda_1)^2}$ is zero we can find  $\lambda_2$ with the desired properties.

\end{proof}

\section{The stabilizer of $\cM_{a,b}$}
\newcommand{\fA}{\mathfrak a}
\newcommand{\fB}{\mathfrak b}
\newcommand{\fC}{\mathfrak c}
\newcommand{\fD}{\mathfrak d}
\newcommand{\fE}{\mathfrak e}
\newcommand{\fF}{\mathfrak f}
\newcommand{\fG}{\mathfrak g}
\newcommand{\fH}{\mathfrak h}
\newcommand{\fI}{\mathfrak i}
\newcommand{\fJ}{\mathfrak j}
\newcommand{\fK}{\mathfrak k}
\newcommand{\fL}{\mathfrak l}
\newcommand{\fM}{\mathfrak m}
\newcommand{\fN}{\mathfrak n}
\newcommand{\fO}{\mathfrak o}
\newcommand{\fP}{\mathfrak p}
\newcommand{\fQ}{\mathfrak q}
\newcommand{\fR}{\mathfrak r}

\label{s:stab}
In this section we shall provide a full description of the stabilizer in $P\Gamma L_{4}(q^{2})$, $q>2$ even,
of the quasi-Hermitian variety $\cM_{a,b}$.
Throughout the
section we shall adopt the notation and the conventions of \cite{ATLAS}.
In particular, $C_m$ is the cyclic group with $m$ elements, while $E_m$
is the elementary Abelian group of order $m$. If $A$ and $B$ are two
groups, we denote by $A\times B$ the direct product of $A$ and $B$, $A.B$ the upward extension
of $A$ by $B$ (i.e. the group $G$ with $A\unlhd G$ such that $G/A\cong B$) and $A:B$
the semidirect product between $A$ and $B$ (where $A$ is
normal in $A:B$ and $B$ acts by conjugation as an automorphism group of $A$).

Let $\phi_s, \psi_{\gamma}, \mu_{\delta}$ where
$s\in\GF{q}, \delta\in\GF{q}^*,
\gamma=(\gamma_1,\gamma_2)\in\GF{q^2}^2$, be the elations associated
with the following non-singular matrices:
\[
  \phi_{s}:
  \begin{pmatrix}
    1 & 0 & 0 & s  \\
    0 & 1 & 0 & 0  \\
    0 & 0 & 1 & 0  \\
    0 & 0 & 0 & 1
  \end{pmatrix}; \qquad
  \psi_{\gamma}(a,b):
  \begin{pmatrix}
    1 & \gamma_1 & \gamma_2 & a(\gamma_1^2+\gamma_2^2)+b(\gamma_1^{q+1}+\gamma_2^{q+1}) \\
    0 & 1        & 0        & (b+b^q)\gamma_1^q                           \\
    0 & 0        & 1        & (b+b^q)\gamma_2^q                           \\
    0 & 0        & 0        & 1
  \end{pmatrix};
\]
\[
  \mu_{\delta}:\text{diag}(1,\delta,\delta,\delta^2).
\]
It follows from \cite[Corollary 4.3]{ACK} that
$$H:=\left\langle \phi_{s},\psi_{\gamma }(a,b),\mu_{\delta}:s\in\GF{q}\textit{, }\gamma \in\GF{q^{2}}^{2}\textit{, }\delta
  \in\GF{q}^{\ast }\right\rangle $$ is a subgroup of $G$ preserving
the Hermitian cone
\[
  \cM_{a,b}\cap \Sigma_{\infty }=\left\{ (0,1,\omega
    ^{i(q-1)},k):i=1,\dots,q+1%
    \textit{, }k\in\GF{q^{2}}^{\ast }\right\}\cup P_{\infty}
\]%
where $\omega$ is a primitive element of $\GF{q^2}$. The subgroup
$$S=\left\langle \phi _{s},\psi_{\gamma }(a,b):s\in\GF{q}\textit{, }\gamma
  \in\GF{q^2}^{2}\right\rangle $$ is a normal Sylow $2$-subgroup of
$H$ and $$ K=\left\langle \phi_{s}:s\in\GF{q}\right\rangle $$ is the
kernel of the action of $H$ on $\Sigma_{\infty }$.
The group
\[ D:=\{\mu_\delta:\delta\in\GF{q}^\ast\}\]
is cyclic of order $q-1$.
Also,
$S$ acts regularly on the $q^{5}$
points of $\cM_{a,b}\setminus \Sigma_{\infty }$.

It can be immediately deduced from \cite[Section 4]{ACK} that the
 induced group $\bar{H}$ on $\Sigma_{\infty }$ is a Frobenius group
 $\bar{H}=\bar{S}:\bar{D}$
 of
order $%
q^{4}(q-1)$ where
\begin{enumerate}
\item $\bar{S}$ is an elementary Abelian $2$-group of order $q^{4}$.
  It is the kernel of $\bar{H}$ and consists of the elations of
  $\Sigma_{\infty }$ with center $P_{\infty }$;	
\item $\bar{D}$ is a  group of $(P_{\infty },m_{\infty })$-homologies of
  $\Sigma _{\infty }$, where $m_{\infty}$ is the line $J=Z=0$.
\end{enumerate}

Thus
$\bar{H}\trianglelefteq \bar{G}\leq \bar{S}:\left( C_{q+1}\times
  \GL{2}{q}\right)$ since the second one is the stabilizer in
$\PGL{3}{q^{2}}$ of the Hermitian cone being $q$ even.

\bigskip

Let $U=\left\{ \tau_{e}:e\in\GF{q}\right\}$, where $\tau_{e}$ is the
elation of $\PGL{4}{q^2}$ represented by the matrix%
\[
  \left( \allowbreak
    \begin{array}{cccc}
      1 & 0   & 0   & 0 \\
      0 & e+1 & e   & 0 \\
      0 & e   & e+1 & 0 \\
      0 & 0   & 0   & 1 %
    \end{array}%
  \right).
\]%
Then $U$ induces on $\Sigma_{\infty }$ a group of
$(V_{\infty },\ell _{\infty })$-homologies, where $V_{\infty
}=(0,1,1,0)$. Further
$(0,1,\omega ^{i(q-1)},k)^{\tau_{e}}=(0,1+\left( 1+\omega
  ^{i(q-1)}\right) e,\omega ^{i(q-1)}+\left( 1+\omega ^{i(q-1)}\right)
e,k)$ with%
\begin{eqnarray*}
  \frac{\omega ^{i(q-1)}+\left( 1+\omega ^{i(q-1)}\right) e}{1+\left( 1+\omega
  ^{i(q-1)}\right) e} &=&\omega ^{i(q-1)}\frac{1+\left( 1+\omega
                          ^{-i(q-1)}\right) e}{1+\left( 1+\omega ^{i(q-1)}\right) e} \\
                      &=&\omega ^{i(q-1)}\frac{1+\left( 1+\omega ^{qi(q-1)}\right) e}{1+\left(
                          1+\omega ^{i(q-1)}\right) e} \\
                      &=&\omega ^{i(q-1)}\left[ 1+\left( 1+\omega ^{i(q+1)}\right) e\right] ^{q-1}
\end{eqnarray*}%
since $e\in\GF{q}$. Thus $\tau_{e}$, and hence $U$, preserves
$\cM_{a,b}\cap \Sigma_{\infty }$ fixing $\ell_{\infty }$. Further, $U$
preserves $%
\cM_{a,b}\setminus \Sigma_{\infty }$. Indeed, if
$(1,x_{0},y_{0},z_{0})\in \cM_{a,b}\setminus \Sigma_{\infty }$,
\begin{eqnarray*}
  z_{0}^{q}+z_{0}+a^{q}(x_{0}^{2q}+y_{0}^{2q})+a(x_{0}^{2}+y_{0}^{2})
  &=&(b^{q}+b)(x_{0}^{q+1}+y_{0}^{q+1})+\\
  &&\qquad (b^{q}+b)\tr(e)(x_{0}+y_{0})^{q+1} \\
  &=&(b^{q}+b)(x_{0}^{q+1}+y_{0}^{q+1})
\end{eqnarray*}
since $e\in\GF{q}$, and so
$(1,x_{0},y_{0},z_{0})^{\tau_{e}}\in \cM_{a,b}\setminus \Sigma_{\infty
}$. Therefore $U=\left\{ \tau_{e}:e\in\GF{q}\right\} $ is an
elementary Abelian $2$-group preserving $%
\cM_{a,b}\setminus \Sigma_{\infty }$. Then $U$ preserves both
$\cM_{a,b}$ and $%
\cB_{a,b}$ since $U$ preserves $\cM_{a,b}\cap \Sigma_{\infty }$ fixing
$\ell _{\infty }$.

\bigskip

\begin{lemma}\label{ker}
  $K$ is the kernel of the action of $G$ on $\Sigma_{\infty }$.
\end{lemma}
\begin{proof}
	The point-wise stabilizer $N$ in $\PGL{4}{q^2}$ of $\Sigma_{\infty }$
	consists of the elations represented by matrices of the
	form%
	\[
	\left( \allowbreak
	\begin{array}{cccc}
		1 & \fB & \fC & \fD \\
		0 & \fA & 0 & 0 \\
		0 & 0 & \fA & 0 \\
		0 & 0 & 0 & \fA%
	\end{array}%
	\right)
	\]%
	with $\fA\neq 0$. Clearly $K\leq N\cap G$. Let $\alpha \in N$ and for each $%
	\theta \in\GF{q^{2}}$ and $\lambda \in\GF{q}$ consider the point $P_{\theta
		,\lambda }=(1,\theta ,\theta ,\lambda )$ in $\cM_{a,b}\setminus
	\Sigma_{\infty }$. Then $P_{\theta ,\lambda }^{\alpha }=(1,\fB+\fA\theta ,\fC+\fA \theta
	,\fA\lambda +\fD)$ which lies in $\cM_{a,b}\setminus \Sigma_{\infty }$ if and
	only if
	\begin{multline*}
		(\fA\lambda +\fD)^{q}+(\fA\lambda
		+\fD)+a^{q}(\fB^{2q}+\fC^{2q})+a(\fB^{2}+\fC^{2})=\\
		(b^{q}+b)(\fB^{q+1}+\fC^{q+1})+(b^{q}+b)(\fB+\fC)\fA^{q}\theta ^{q}+(b^{q}+b)(\fB+\fC)^{q}\fA\theta
	\end{multline*}
	is satisfied for each $\theta \in\GF{q^{2}}$ and $\lambda \in\GF{q}$. Thus $\fB=\fC$ and $%
	\fA,\fD\in\GF{q}$ since $\fA\neq 0$.
	
	{\normalsize Now, let $Q=(1,x_{0},y_{0},z_{0})$ in $\cM_{a,b}\setminus \Sigma
		_{\infty }$, then%
		\[
		z_{0}^{q}+z_{0}+a^{q}(x_{0}^{2q}+y_{0}^{2q})+a(x_{0}^{2}+y_{0}^{2})=(b^{q}+b)(x_{0}^{q+1}+y_{0}^{q+1})%
		\mathit{.}
		\]%
		Now, $Q^{\alpha }=(1,\fB+\fA x_{0},\fB+\fA y_{0},\fD+\fA z_{0})$ lies in in $%
		\cM_{a,b}\setminus \Sigma_{\infty }$ if and only if%
		\begin{multline*}
			\fA \tr(z_{0})+\fA^{2}a^{q}(x_{0}^{2q}+y_{0}^{2q})+\fA^{2}a(x_{0}^{2}+y_{0}^{2})=
			\\ \fA^{2}(b^{q}+b)(x_{0}^{q+1}+y_{0}^{q+1})+(b^{q}+b)\fA \tr((x_{0}+y_{0})\fB^{q})
		\end{multline*}
		and hence %
		\begin{equation}
		    \label{lakunoc}
		(\fA+1)\tr(z_{0})=(b^{q}+b)\tr((x_{0}+y_{0})\fB^{q})
		\end{equation}
		since $Q\in \cM_{a,b}\setminus \Sigma_{\infty }$. We may repeat the previous
		argument by choosing two distinct points $Q_{i}=(1,x_{i},x_{i}+\mu ,z_{i})$
		with $\tr(x_{1})\neq \tr(x_{2})$, $\mu \in\GF{q}^{\ast }$ and $z_{i}$ such that $%
		\tr(z_{i})=\mu ^{2}\tr(a)+\tr(b)(\tr(x_{i})\mu+\mu^{2}) $. Therefore, $\tr(z_{1})\neq \tr(z_{2})$ since $b \notin \GF{q}$,
		and $Q_{1},Q_{2}\in \cM_{a,b}\setminus \Sigma_{\infty }$. Since $Q_{i}^{\alpha
		}\in \cM_{a,b}\setminus \Sigma_{\infty }$, we may argue as above with $Q_{i}$ in the role of $Q$ and hence ((\ref{lakunoc})) becomes%
		\[
		(\fA+1)\tr(z_{i})=(b^{q}+b)\mu \tr(\fB^{q}) \textit{,}
		\]%
		which leads to $\fA=1$ and $\fB\in\GF{q}$. Now, let $R=(1,x_{0},y_{0},z_{0})$ with $%
		x_{0}+y_{0}\notin\GF{q}$. Then $R^{\alpha }\in \cM_{a,b}\setminus \Sigma
		_{\infty }$ implies $\fB \tr(x_{0}+y_{0})=0$. Therefore $\fB=0$ as $%
		x_{0}+y_{0}\notin\GF{q}$. Thus $G\cap N=K$, which is the assertion. }
\end{proof}

\bigskip

\begin{proposition}
\label{p:sylow}
$S:U$ is a Sylow $p$-subgroup of $G$.
\end{proposition}

\begin{proof}
	Clearly, $U\cap S=1$. It is easy to check that
	\begin{eqnarray*}
		\tau_{e}\phi_{s} &=&\phi_{s}\tau_{e} \\
		\tau_{e}^{-1}\psi_{\gamma }(a,b)\tau_{e} &=&\psi_{\gamma ^{\prime }}(a,b)\textit{,
			where }\gamma ^{\prime }=(\gamma_{1}+\left( \gamma_{1}+\gamma_{2}\right)
		e,\gamma_{2}+\left( \gamma_{1}+\gamma_{2}\right) e)
	\end{eqnarray*}%
	and hence $S:U$ lies in a Sylow $p$-subgroup $W$ of $G$. Note that $%
	W=SW_{O} $, where $O=(1,0,0,0)$, since $S$ acts transitively on $%
	\cM_{a,b}\setminus \Sigma_{\infty }$. Let $\alpha \in W_{O}$, $\alpha $
	is represented by
	\[
	\left(
	\begin{array}{cccc}
		1 & 0 & 0 & 0 \\
		0 & \fF+1 & \fF & \fG+\fE \\
		0 & \fF & \fF+1 &\fE \\
		0 & 0 & 0 & 1%
	\end{array}%
	\right)
	\]%
	for suitable $\fE,\fF,\fG\in\GF{q^2}$ as a consequence of Lemma \ref{autaut0}. Let $(1,x_{0},y_{0},z_{0})\in
	\cM_{a,b}\setminus \Sigma_{\infty }$, then%
	\[
	(1,x_{0},y_{0},z_{0})^{\alpha
	}=(1,x_{0}+\fF(x_{0}+y_{0}),y_{0}+\fF(x_{0}+y_{0}),z_{0}+\fE(x_{0}+y_{0})+\fG x_{0})
	\]%
	which actually lies in $\cM_{a,b}$ if and only if
	\[
	\tr\left( \fE(x_{0}+y_{0})+\fG x_{0}\right) =(b^{q}+b)\left[
	\tr(x_{0}^{q}\fF(x_{0}+y_{0}))+\tr(y_{0}^{q}\fF(x_{0}+y_{0}))\right]
	\]%
	and hence
	\[
	\tr\left( \fE(x_{0}+y_{0})+\fG x_{0}\right) =(b^{q}+b)(x_{0}+y_{0})^{q+1}\tr(\fF)
	\]%
	Since $(1,\omega ^{i},\omega ^{i},1)\in \cM_{a,b}$ for each $i=0,\dots,q^{2}-2$ given $\omega$ a primitive element of $\GF{q^2}$,
	it follows that $(1,\omega ^{i},\omega ^{i},1)^{\alpha}\in \cM_{a,b}$ if and
	only if $\tr(\fG\omega ^{i})=0$, hence $\fG=0$.
	
	Let $\lambda \in\GF{q^2}$ and $z_{\lambda }\in\GF{q^2}$ such that $%
	\tr(z_{\lambda })=\tr(a\lambda ^{2} )+(b^{q}+b)(\omega ^{i(q+1)}+(\lambda -\omega
	^{i})^{q+1})$, then $(1,\omega ^{i},\lambda -\omega ^{i},z_{\lambda })\in
	\cM_{a,b}$ and so $(1,\omega ^{i},\lambda -\omega ^{i},z_{\lambda })^{\alpha
	}\in \cM_{a,b}$ if and only if $\tr\left( \lambda \fE\right) =(b^{q}+b)\lambda
	^{q+1}\tr(\fF)$. Now, choosing distinct $\lambda $'s in $\GF{q}^{\ast }$ we
	obtain $\tr\left( \fE\right) =\tr(\fF)=0$ and so $\fE,\fF\in\GF{q}$. Therefore, $%
	\fE \tr(x_{0}+y_{0})=0$. Finally, choosing $\lambda $ in $\GF{q^{2}}\setminus\GF{q}$, we
	obtain $\fE=0$. Thus $\alpha \in U$ and so $W=S:U$.
\end{proof}

\begin{proposition}
  \label{WnormG}Let $W=S:U$, then $W\vartriangleleft G$.
\end{proposition}

\begin{proof}
  Assume that $W$ is not normal in $G$. Then there
  is $g\in G$ such that $%
  W^{g}\neq W$. Nevertheless, $S\leq W^{g}$ since $K\leq S$,
  $K\vartriangleleft G$ by Lemma \ref{ker}, and since $S/K\cong E_{q^{4}}$
  induces the (full) elation group of center $%
  P_{\infty }$ on $\Sigma_{\infty }$. Thus the group induced by
  $\left\langle W^{g},W\right\rangle $ on the Hermitian cone
  $\cM_{a,b}\cap \Sigma_{\infty }$ contains two distinct Sylow
  $p$-subgroups and hence contains $SL_{2}(q)$ since
  $\bar{G}\leq \bar{S}:\left( C_{q+1}\times
    \GL{2}{q}\right) $ and $q>2$, where $\bar{G}$ is the the group induced on $\cM_{a,b}\cap \Sigma_{\infty }$ by $G$. Let $R$ be subgroup of
  $\left\langle W^{g},W\right\rangle $ inducing a cyclic subgroup of
  $SL_{2}(q)$ of order $q+1$. Note that $R$ can be
  chosen in a way that $R\cap S=1$ since $S$ is a $p$-group. Then $R$
  fixes a point $\cM_{a,b}\setminus \Sigma_{\infty }$ since
  $\left\vert \cM_{a,b}\setminus \Sigma_{\infty }\right\vert =q^{5}$
  and permutes regularly the $q+1$ lines of the Hermitian cone. Since
  $S\vartriangleleft G$ and $S$ acts transitively on
  $\cM_{a,b}\setminus \Sigma_{\infty }$, possibly substituting $R$
  with a suitable conjugate in $S:R$, we may assume that that $R$
  fixes $O=(1,0,0,0)$. Now, also $U$ fixes $O$ thus
  $\left\langle U,R\right\rangle $ fixes $O$ and hence
  $\left\langle U,R\right\rangle \cap S=1$. Thus
  $\left\langle U,R\right\rangle $ acts faithfully on the Hermitian
  cone and contains subgroups of order $q$ and $q+1$, and the last one
  acts transitively on the lines of the Hermitian cone. Thus
  $\left\langle U,R\right\rangle $ contains a copy of $SL(2,q)$. Then
  $\left\langle U,R\right\rangle $ contains a conjugate of $R$, say
  $\left\langle \zeta \right\rangle $ with $\zeta $ represented by the
  matrix $Diag(1,\omega ^{q-1},\omega ^{1-q},1)$ with $\omega$ a primitive element of $\GF{q^2}^{\ast}$.
	
  Let $P_{\theta ,\lambda }=(1,\theta ,\theta ,\lambda )$ with
  $\theta \in\GF{q^{2}}$ such that $\tr(a\theta ^{2})\neq 0$ and
  $\lambda \in\GF{q}$. Then $%
  P_{\theta ,\lambda }$ lies in $\cM_{a,b}\setminus \Sigma_{\infty }$
  and hence
  $P_{\theta ,\lambda }^{\zeta }=(1,\omega ^{q-1}\theta ,\omega
  ^{1-q}\theta ,\lambda )$ must lie in
  $\cM_{a,b}\setminus \Sigma_{\infty }$.

  Thus
  \begin{align*}
      \tr(a\theta^2(\omega^{2(q-1)}+\omega^{2(1-q)}))&=0\\
      \tr\left(a\theta^2\frac{\omega^{4q}+\omega^4}{\omega^{2(q+1)}} \right)=\tr(a\theta^2)\frac{\tr(\omega^4)}{\nr(\omega^2)}&=0
  \end{align*}
  and so $\tr(a\theta^2)=0$ because $\tr(\omega^4)\neq 0$. Thus $W$ is normal in $G$.
\end{proof}

\medskip

\begin{theorem}
  \label{LinColl}$G=\left\langle
    \phi_{s},\psi_{\gamma }(a,b),\tau_{e},\mu _{\delta }:\gamma
    \in\GF{q^2}^{2}\mathit{, }s,e,\delta \in\GF{q}\mathit{, }%
    \delta \neq 0\right\rangle $. It has order $q^{6}(q-1)$.
\end{theorem}

\begin{proof}
  First, we observe that
  $E_{q^{4}}:\left( C_{q-1}\times E_{q}\right) =\left\langle
    W,D\right\rangle /K\leq G/K$ and hence
  $E_{q^{4}}:\left( C_{q-1}\times E_{q}\right) \trianglelefteq
  G_{\Sigma_{\infty }}/K\trianglelefteq E_{q^{4}}:\left(
    C_{q-1}.\left( E_{q}:C_{q-1}\right) \right) $.
	
  Assume that there is an element of odd order $\varrho \ $in
  $G$ such that
  $\bar{\varrho}\notin E_{q^{4}}:\left( C_{q-1}\times E_{q}\right)
  $. Then $\bar{\varrho}$ preserves $\ell_{\infty }$ and fixes
  $P_{\infty }$, and two further points since $o(\bar{\varrho%
  })\mid q-1$, namely one on
  $\ell_{\infty }\setminus \left\{ P_{\infty }\right\} $ and the other
  on
  $\left( \cM_{a,b}\cap \Sigma _{\infty }\right) \setminus
  \ell_{\infty }$. Recall that $S/K$ is the group of $(P_{\infty},P_{\infty})$-elations of $\Sigma_{\infty}$, then it acts transitively on $\ell_{\infty } \setminus \{P_{\infty }\}$, and hence we may assume that $\bar{\varrho}$ fixes the point $\{V_{\infty}\}=\ell_{\infty } \cap m_{\infty}$, where $m_{\infty}$ is the line $J=Z=0$. Moreover, the stabilizer in $S/K$ of $V_{\infty}$ acts regularly on the set $q^{2}$ lines of $\Sigma_{\infty}$ which are incident with $V_{\infty}$ and are distinct from $\ell_{\infty}$, thus we may also assume that $\bar{\varrho}$ preserves $m_{\infty}$. Therefore, $\left\langle D/K,\bar{%
      \varrho}\right\rangle $ is a subgroup of the stabilizer of a
  triangle in $\Sigma_{\infty }$ with $P_{\infty }$ and $V_{\infty}$ as two of its vertices and $\ell_{\infty}$ and $m_{\infty}$ as two of its sides. Actually,
  $\left\langle D/K,\bar{\varrho}%
  \right\rangle \leq C_{q-1}\times C_{q-1}$, where the group
  $C_{q-1}\times C_{q-1}$ is generated by cyclic subgroup of
  homologies in a triangular configuration. Since $D/K$ is a cyclic
  group of order $q-1$ consisting of $(P_{\infty },m_{\infty })$-homologies of
  $\Sigma _{\infty }$, by suitably multiplying $\bar{\varrho}$ with an
  element of $D/K$ we may assume that $%
  \bar{\varrho}$ is a $(Q,\ell_{\infty })$-homology of
  $\Sigma_{\infty }$, where $Q=(0,1,w,0)$ with $w$ a fixed element of $\GF{q^{2}} \setminus \{1\}$
  such that $\nr(w)=1$. Thus, $\varrho $ fixes $\ell_{\infty }$
  point-wise, fixes $Q$ and preserves the points on the Hermitian
  cone $\cM_{a,b}\cap \Sigma_{\infty } $ with apex $P_{\infty }$ and
  $\cM_{a,b}\setminus \Sigma_{\infty }$. In particular, $\varrho $
  preserves
  $\cM_{a,b}\cap m_{\infty }=\left\{ (0,1,x,0):\nr(x)=1\right\} $.

  Now, we
  are going to determine the matrix representation of $\varrho$.
By Remark \ref{remNot} we have

   \[
     \begin{pmatrix}
       0 & 1 & 1 & 0 %
     \end{pmatrix}
     \begin{pmatrix}
       1 & \fB & \fC & \fD \\
       0 & \fF{} & \fG{} & \fH{} \\
       0 & \fJ{} & \fK{} & \fL{} \\
       0 & 0 & 0 & \fP{} %
     \end{pmatrix}=%
     \begin{pmatrix}
       0 & \fF{}+\fJ{} & \fG{}+\fK{} & \fH{}+\fL{} %
     \end{pmatrix},
   \]%
   \begin{multline*}
     \begin{pmatrix}
       0 & 1 & w & 0 %
     \end{pmatrix}
     \begin{pmatrix}
       1 & \fB     & \fC & \fD \\
       0 & \fF{}     & \fG{} & \fH{} \\
       0 & \fF{}+\fG{}+\fK{} & \fK{} & \fH{} \\
       0 & 0     & 0 & \fP{} %
     \end{pmatrix}=\\
     \begin{pmatrix}
       0 & \fF\left( w+1\right)+\fG{}w+\fK{}w & \fG{}+\fK{}w & \fH{}\left( w+1\right) %
     \end{pmatrix}.
   \end{multline*}
   Since the collineation $\rho$ exists and $w\neq 1$ we have $\fF\left( w+1\right)+\fG{}w+\fK{}w\neq0$,
   $\fG+\fK w\neq 0$ and $\fH=0$.
   Furthermore,
   \begin{multline*}
     \left( \fF{}+\fF{}w+\fG{}w+\fK{}w\right) w+\fG{}+\fK{}w = \fF{}\left( w^2+w\right) +\fG{}\left(
       w^{2}+1\right) +\fK{}\left( w^{2}+w\right)=0
   \end{multline*}%
   \[
     \fK{}=\fF{}+\fG{}\left( w^{q}+1\right).
   \]%
  In particular
 $\fF \neq \fG,\fG w^{q}$ because the matrix associated to $\varrho$ is non-singular.
   Consider the point of $ \cM_{a,b}$ with coordinates $(1,1,1,\theta)$, where $\theta \in \GF{q}$. Then {\small
     \[
       \left(
         \begin{array}{cccc}
           1 & 1 & 1 & \theta
         \end{array}%
       \right) \left(
         \begin{array}{cccc}
           1 & \fB      & \fC                        & \fD \\
           0 & \fF      & \fG                        & 0 \\
           0 & \fG w^{q} & \fF + \fG \left( w^{q}+1\right) & 0 \\
           0 & 0      & 0                        & \fP %
         \end{array}%
       \right) =\left( \allowbreak
         \begin{array}{cccc}
           1 & \fB+\fF+\fG w^{q} & \fC+\fF+\fG w^{q} & \fD+\fP \theta
         \end{array}%
       \right),
     \]%
   }%
   \begin{equation}
     \tr(\fD)+\tr(\fP) \theta +\tr(a(\fB+\fC)^{2}=\tr(b)[\fB^{q+1}+\fC^{q+1}+\tr(\fB+\fC)(\fF^{q}+\fG^{q}w)].
     \label{dva}
   \end{equation}%
   which must be fulfilled for each $\theta \in \GF{q}$ and hence $\tr(\fP)=0$ and
   \begin{equation*}
     \tr(\fD)+\tr(a\fB+\fC)^{2}=\tr(b)[\fB^{q+1}+\fC^{q+1}+\tr\left( (\fB+\fC)(\fF^{q}+\fG^{q}w)\right) ] \textit{.}
   \end{equation*}
  Thus, $%
   \fP\in\GF{q}^*$.

   Now $\varrho $ lies $G$, so
   $\mu_{\delta }\varrho$ does. Now, possibly after choosing
   $\delta =\fP^{-1/2}$ since $\fP\in\GF{q}$, we may consider
   \[
     \varrho ^{\prime }=\mu_{\fP^{-1/2}}\varrho =
     \begin{pmatrix}
       1 & \fB      & \fC                        & \fD \\
       0 & \fF      & \fG                        & 0 \\
       0 & \fG w^{q} & \fF+ \fG \left( w^{q}+1\right) & 0 \\
       0 & 0      & 0                        & 1 %
     \end{pmatrix},
   \]%
   which no longer induces a $(Q,\ell_{\infty})$-homology of
   $\Sigma _{\infty }$, but still has odd order and fixes the triangle vertices $P_{\infty}$,
   $V_{\infty}$ and $Q$, and preserves the points on the Hermitian
   cone $%
   \cM_{a,b}\cap \Sigma_{\infty }$ with apex $P_{\infty}$ and
   $\cM_{a,b}\setminus \Sigma_{\infty}$. In particular,
   $\varrho ^{\prime }$ preserves $%
   \cM_{a,b}\cap m_{\infty }=\left\{ (0,1,y,0):\nr(y)=1\right\} $, thus
   \begin{multline*}
     \left(
       \begin{array}{cccc}
         0 & 1 & y & 0 %
       \end{array}%
     \right) \left(
       \begin{array}{cccc}
         1 & \fB      & \fC                        & \fD \\
         0 & \fF      & \fG                        & 0 \\
         0 & \fG w^{q} & \fF+ \fG\left( w^{q}+1\right) & 0 \\
         0 & 0      & 0                        & 1 %
       \end{array}%
     \right) = \\
     \begin{pmatrix}
         0 & \fF+ \fG w^{q}y & \fG+ \fF y+\fG y+\fG w^{q}y & 0 %
       \end{pmatrix},%
   \end{multline*}%
with $\fF \neq \fG w^{q}y$ for each $y \in \GF{q^{2}}$ with $N(y)=1$, and hence
	\[
          \nr\left( \frac{\left( \fF+\fG+\fG w^{q}\right)
              y+\fG}{\fG w^{q}y+\fF}\right) =1.
	\]%
	Thus,
        $\varsigma :x\mapsto \frac{\left( \fF+ \fG +\fG w^{q}\right)
          x+\fG}{\fG w^{q}x+\fF}$ is an element of $\PGL{2}{q^{2}}$ fixing $%
        1,w$ and preserving the Baer subline
        $\left\{  y \in \GF{q^{2}}:\nr(y)=1\right\} $ of $PG_{1}(q^{2})$. Hence, $%
        \varsigma $ lies in the cyclic subgroup of order $q-1$ of
        $\PGL{2}{q^{2}}$ fixing $%
        1,w$ and preserving $\left\{  y \in \GF{q^{2}}:\nr(y)=1\right\} $.
	
	For each $c\in\GF{q}^{\ast }$ consider $\alpha_{c}\in
        \PGL{2}{q^{2}}$ defined by
	\[
          \alpha_{c}:x\mapsto \frac{\allowbreak x\left(
              c+cw+w^{q}+1\right) +c+cw+w+1}{\allowbreak x\left(
              c+cw^{q}+w^{q}+1\right) +c+cw^{q}+w+1}.
	\]%
	Then $1^{\alpha_c }=1$, $w^{\alpha_c }=w$ and
        $(w^{q})^{\alpha_c }=w\left(w(c+cw^{q}+w+1)+c+cw^{q}+w^{q}+1 \right)^{q-1}$. Indeed,
	\begin{eqnarray*}
          \frac{w^{q}\left( c+cw+w^{q}+1\right) +c+cw+w+1}{\allowbreak w^{q}\left(
          c+cw^{q}+w^{q}+1\right) +c+cw^{q}+w+1} &=&\frac{\left(w(c+cw^{q}+w+1)+c+cw^{q}+w^{q}+1 \right)^{q}}{w^{q} \left(w(c+cw^{q}+w+1)+c+cw^{q}+w^{q}+1 \right) }\\
                                                 &=&w\left(w(c+cw^{q}+w+1)+c+cw^{q}+w^{q}+1 \right)^{q-1}\mathit{.}
	\end{eqnarray*}%
	Thus $\left\{ \alpha_{c}:c\in\GF{q}^{\ast }\right\} $ is the
        cyclic subgroup of order $q-1$ of $\PGL{2}{q^2}$ fixing $%
        1,w$ and preserving the Baer subline $\left\{ y \in \GF{q^{2}}:\nr(y)=1\right\} $ of $PG_{1}(q^{2})$, and hence%
	\begin{eqnarray*}
          \fF &=&\allowbreak \frac{c+w+cw^{q}+1}{d}, \\
          \fG &=&\frac{c+cw+w+1}{d}
	\end{eqnarray*}%
	for some suitable $c\in\GF{q}^{\ast }$ and
        $d \in\GF{q^2}^{\ast } $. It is easy to check that, such $\fF$ and $\fG$ are such that $\nr(\fF) \neq \nr(\fG)$ and hence $\fF$ and $\fG$ fulfill $\fF \neq \fG w^{q}y$ for each $y \in \GF{q^{2}}$ with $N(y)=1$. Further,
	$$\fF+\fG w^{q} + \fG =\frac{c+w+cw^{q}+1}{d}+\left( \frac{c+cw+w+1}{d}\right) w^{q}+\frac{c+cw+w+1}{d}=\frac{c+cw+w^{q}+1}{d}$$ and by multiplying each term of the matrix representing
        $\varrho^{\prime}$ by $d$, we may assume that
        $\varrho^{\prime}$ is represented by
	$$
		\begin{pmatrix}
                  d & \fB                        & \fC            & \fD \\
                  0 & \allowbreak c+w+cw^{q}+1 & c+cw+w+1     & 0 \\
                  0 & c+cw^{q}+w^{q}+1         & c+cw+w^{q}+1 & 0 \\
                  0 & 0                        & 0            & d %
		\end{pmatrix}.
	$$
	Note that $\tau_{\frac{c+1}{w+w^{q}}}\varrho ^{\prime }$ is an
        element of $G$ represented by
	\begin{multline*}
          \begin{pmatrix}
            1 & 0                     & 0                     & 0 \\
            0 & \frac{c+1}{w+w^{q}}+1 & \frac{c+1}{w+w^{q}}   & 0 \\
            0 & \frac{c+1}{w+w^{q}}   & \frac{c+1}{w+w^{q}}+1 & 0 \\
            0 & 0                     & 0                     & 1 %
          \end{pmatrix}
          \begin{pmatrix}
            d & \fB                        & \fC            & \fD \\
            0 & \allowbreak c+w+cw^{q}+1 & c+cw+w+1     & 0 \\
            0 & c+cw^{q}+w^{q}+1         & c+cw+w^{q}+1 & 0 \\
            0 & 0                        & 0            & d %
          \end{pmatrix}=\\
          \begin{pmatrix}
            d & \fB            & \fC        &\fD \\
            0 & w+cw^{q}     & w+cw     & 0 \\
            0 & cw^{q}+w^{q} & cw+w^{q} & 0 \\
            0 & 0            & 0        & d %
          \end{pmatrix}.
	\end{multline*}
	Set $\fR=w+cw^{q}$, then $\tau_{\frac{c+1}{w+w^{q}}}\varrho
        ^{\prime }$ is of the form
	\[
          \begin{pmatrix}
              d & \fB             & \fC         & \fD \\
              0 & \fR             & \fR+\lambda & 0 \\
              0 & \fR^{q}+\lambda & \fR^{q}     & 0 \\
              0 & 0             & 0         & d %
           \end{pmatrix},
	\]%
	where $\lambda =c\tr(w)\neq 0$ and $\tr(\fR)=\tr(w)(c+1)=\lambda
        +\tr(w)$.
	
	The point $(1,x_{0},x_{0},c_{0})$ with $c_{0}\in\GF{q}$ lies in $\cM_{a,b}\setminus
        \Sigma_{\infty }$, and hence
	\begin{multline*}
          \begin{pmatrix}
            1 & x_{0} & x_{0} & c_{0} %
          \end{pmatrix}
          \begin{pmatrix}
            d & \fB             & \fC         & \fD \\
            0 & \fR             & \fR +\lambda & 0 \\
            0 & \fR^{q}+\lambda & \fR^{q}     & 0 \\
            0 & 0             & 0         & d %
          \end{pmatrix}=\\
          \begin{pmatrix}
            d & \fB+x_{0}\lambda +\fR^{q}x_{0}+\fR x_{0} & \fC+x_{0}\lambda + \fR^{q}x_{0}+\fR x_{0} & c_{0}d+\fD %
          \end{pmatrix},
	\end{multline*}
	which is equivalent to
	$$
		\begin{pmatrix}
                  1 & \frac{\fB+x_{0}\lambda +\fR^{q}x_{0}+\fR x_{0}}{d} & \frac{\fC +x_{0}\lambda +\fR^{q}x_{0}+\fR x_{0}}{d} & c_{0}+\frac{\fD}{d} %
		\end{pmatrix} \textit{,}
	$$
 where
	\[
          \left( \frac{\fB +x_{0}\lambda +\fR^{q} x_{0}+\fR x_{0}}{d}\right)
          +\left( \frac{\fC +x_{0}\lambda +\fR^{q} x_{0}+\fR x_{0}}{d}\right)
          =\frac{1}{d}\left(\fB+\fC \right),
	\]
	\begin{eqnarray*}
          \frac{1}{d^{q+1}} \left[ \left( \fB+x_{0}\lambda +\fR^{q}x_{0}+\fR x_{0}\right) ^{q+1}+\left(\fC+x_{0}\lambda
          +\fR^{q}x_{0}+\fR x_{0}\right)^{q+1} \right] &= \\ \frac{1}{d^{q+1}} \left[ \fB^{q+1}+\fC^{q+1}+(\lambda +\tr(\fR))\tr[(\fB+\fC)x_{0}^{q}] \right] &= \\
          \frac{1}{d^{q+1}} \left[\fB^{q+1}+\fC^{q+1}+\tr(w)\tr[(\fB+\fC)x_{0}^{q}] \right] \textit{.}
	\end{eqnarray*}
 Now, since the image of the point $\left(1,x_{0},x_{0},c_{0} \right)$ under $\tau_{\frac{c+1}{w+w^{q}}}\varrho
        ^{\prime }$ must lie in $\cM_{a,b}\setminus
        \Sigma_{\infty }$, it follows that
	\begin{eqnarray*}
          \tr(c_{0})+\tr(\fD /d)+\tr(a(\fB/d+\fC/d)^{2})=& \\
          \tr(b)[(\fB/d)^{q+1}+(\fC/d)^{q+1}]+\tr(b)\tr(w)\tr[(\fB+\fC)x_{0}^{q}]/d^{q+1}, \\[3pt]
          \tr(\fD/d)+\tr(a(\fB/d+\fC/d)^{2})= \\
          \tr(b)[(\fB/d)^{q+1}+(\fC/d)^{q+1}]+\tr(b)\tr(w)\tr[(\fB+\fC)x_{0}^{q}]/d^{q+1}.
	\end{eqnarray*}
The previous equation must be fulfilled for each value of $x_{0}$ in $\GF{q^{2}}$, then $\tr(b)\tr(w)(\fB+\fC)=0$ and
	\begin{equation}\label{krug}
          \tr(\fD/d)+\tr(a(\fB/d+\fC/d)^{2}) =\tr(b)\left[(\fB/d)^{q+1}+(\fC/d)^{q+1}\right] \textit{.}
	\end{equation}
	Therefore $\fB=\fC$ since $b,w \in \GF{q^{2}}\setminus \GF{q}$, and hence $\fD /d \in\GF{q}$. Thus

	\[
          \varrho ^{\prime \prime }= \phi_{ \fD /d} \tau_{\frac{c+1}{w+w^{q}}}\varrho ^{\prime }=
              \begin{pmatrix}
              d & \fB      & \fB          & 0 \\
              0 & \fR      & \fR+\lambda          & 0 \\
              0 & \fR^{q}+\lambda & \fR^{q} & 0 \\
              0 & 0      & 0          & d %
            \end{pmatrix}
	\]
	lies in $G$, and hence it
	preserves $\cM_{a,b}\setminus \Sigma_{\infty }$.

                     Since the trace is surjective, any
                     point with coordinates $(1,x,0,z)$, where $x$ is any element of $\GF{q^2}$, and $z$ is a suitable element of $\GF{q^2}$ depending on the choice of $x$,
                    lies in $\cM_{a,b}\backslash \Sigma_{\infty
                     }$, and hence
                     \begin{equation}\label{bucurest}
                       z^{q}+z+a^{q}x^{2q}+ax^{2}
                       =(b^{q}+b)x^{q+1} \textit{,}
                     \end{equation}
                     Then $(1,x,0,z)^{\varrho ^{\prime \prime }}$, which is given by
                     \begin{multline*}
          \begin{pmatrix}
            1 & x & 0 & z %
          \end{pmatrix}\begin{pmatrix}
                         d & \fB             & \fB         & 0 \\
                         0 & \fR             & \fR+\lambda & 0 \\
                         0 & \fR^{q}+\lambda & \fR^{q}     & 0 \\
                         0 & 0             & 0         & d %
                       \end{pmatrix}=
                       \begin{pmatrix}
                         d & \fB+\fR x & \fB+x\lambda +\fR x & zd %
                       \end{pmatrix} \textit{,}
                     \end{multline*}
                     lies in $
                     \cM_{a,b}\backslash \Sigma_{\infty }$, and hence
                     \begin{equation}\label{Chisinau}
                       d^{q+1}(z^{q}+z)+a^{q}\lambda ^{2} d^{1-q}x^{2q}+a\lambda ^{2} d^{q-1}x^{2}
                       = (b^{q}+b)\left[ \lambda ^{2} x^{q+1}+\lambda \tr\left[ \left(
                             \fB + \fR x\right) x^{q}\right] \right],
                     \end{equation}
                     Now, combining (\ref{bucurest}) with (\ref{Chisinau}) one obtains
                     \begin{multline}
                       a^{q}\left( \lambda ^{2}d^{1-q}+d^{q+1}\right) x^{2q}+a\left( \lambda
                         ^{2}d^{q-1}+d^{q+1}\right) x^{2}=\\ (b^{q}+b)\left[ \left( \lambda ^{2}+d^{q+1}+\lambda \tr (\fR))
                         x^{q+1}+\lambda \tr\left( \fB x^{q} \right)\right) %
                       \right].  \label{jedan}
                     \end{multline}
                     Consequently, equality in (\ref{jedan}) must be fulfilled for each $x \in \GF{q^2}$. Thus $\lambda ^{2}d^{1-q}+d^{q+1}=\lambda
                         ^{2}d^{q-1}+d^{q+1}=\fB= \tr (\fR)=0$, and hence $c=1$, $\fR=\tr(w)=\lambda=d$ since $\lambda =c\tr(w)\neq 0$, $\fR=w+cw^{q}$ and $\tr(\fR)=\tr(w)(c+1)=\lambda
        +\tr(w)$, where $w \in\GF{q^{2}} \setminus \{1\}$ is such that $\nr(w)=1$. So $\varrho ^{\prime \prime }=1$, and hence $\varrho=\left( \phi_{\fD/d}\tau_{\frac{%
              c+1}{w+w^{q}}}\mu_{\fP^{-1/2}}\right)^{-1} \in \left\langle W,D\right\rangle$, which is a contradiction.
                   \end{proof}

\bigskip

From now on, we denote the stabilizer in $PGL_{4}(q^{2})$ and in $P\Gamma L_{4}(q^{2})$
    of $\cM_{a,b}$ by $G(a,b)$ and $\Gamma(a,b)$ respectively.

\bigskip

\begin{theorem}
  \label{SemiLinColl} Let $\sigma$ be an element $P\Gamma L_{4}(q^{2})$ induced by a generator of $\mathrm{Aut}(\GF{q^2})$, and let
$\beta$ be an element $P\Gamma L_{4}(q^{2})$ of the form as in Lemma \ref{fin} mapping $\cM_{1,\epsilon}$ with $Tr(\epsilon)=1$ onto $\cM_{a,b}$, which exists by Theorem \ref{mainequiv}. Then
  $$\Gamma(a,b)=\left\langle
    \phi_{s},\psi_{\gamma }(a,b),\tau_{e},\mu _{\delta },\sigma^{\beta}:\gamma
    \in\GF{q^2}^{2}\mathit{, }s,e,\delta \in\GF{q}\mathit{, }%
    \delta \neq 0\right\rangle \textit{,}$$
  and its order is $q^{6}(q-1)\log_{2}q$.
\end{theorem}

\begin{proof}
  We may assume that $\sigma :(j,x,y,z)\mapsto
  (j^{2},x^{2},y^{2},z^{2})$. Clearly, $%
  \sigma $ fixes $\Sigma_{\infty }$, $m_{\infty }$ $P_{\infty }$,
  where $%
  m_{\infty }:J=Z=0$. Also $\sigma $  permutes the points
  $(0,1,\omega ^{j},0)$, where $j=0,\dots,q-1$ fixing $(0,1,1,0)$. Thus,
  $\left\langle \sigma \right\rangle $ preserves the Hermitian cone
  {\normalsize $\cM_{1,\epsilon}\cap \Sigma_{\infty }$} fixing
  $\ell_{\infty }$.
	
	Now, let $(1,x,y,z)\in \cM_{1,\epsilon}\setminus \Sigma_{\infty}$
	then
	\[
		\tr (z)+\tr (x^{2}+y^{2})+N(x)+N(y) =0,
	\]
	thus
	\[
		\left(\tr (z)+\tr (x^{2}+y^{2})+N(x)+N(y) \right)^{2^{i}} =0,
	\]
	and hence
	\[
		(\tr (z^{2^{i}})+\tr ((x^{2^{i}})^{2}+(y^{2^{i}})^{2})+N(x^{2^{i}})+N(y^{2^{i}}) =0 \textit{.}
	\]
	Therefore, $\sigma$ preserves $\cM_{1,\epsilon}$  and hence $\Lambda(1,\varepsilon) \leq \Gamma(1,\varepsilon)$, where
	$$ \Lambda(1,\varepsilon)=G(1,\varepsilon)\left\langle \sigma\right\rangle = \left\langle
    \phi_{s},\psi_{\gamma }(1,\epsilon),\tau_{e},\mu _{\delta },\sigma:\gamma
    \in\GF{q^2}^{2}\mathit{, }s,e,\delta \in\GF{q}\mathit{, }%
    \delta \neq 0\right\rangle \textit{.}$$
	Let $\xi \in \Gamma (1,\varepsilon)$, then $\xi \in P\Gamma L_{4}(q^{2})$ and hence $\xi=\sigma ^{j}\alpha$ for some $j=0,...,\log_{2}q-1$ and $\alpha \in PGL_{4}(q^{2})$. Then $\sigma ^{-j} \xi \in\Gamma(1,\varepsilon) \cap PGL_{4}(q^{2})= G(1,\varepsilon)$ by Theorem \ref{LinColl} since $\sigma$ preserves $\cM_{1,\epsilon}$. Thus, $\xi \in G(1,\varepsilon)\left\langle \sigma\right\rangle =\Lambda(1,\varepsilon)$, and hence $ \Lambda(1,\varepsilon)= \Gamma(1,\varepsilon)$, whose order clearly is $q^{6}(q-1)\log_{2}q$. Then $\Gamma(a,b)$ has order $q^{6}(q-1)\log_{2}q$ since $\Gamma(a,b)=\Gamma(1,\varepsilon)^{\beta}$. Therefore,
	$$\Gamma(a,b)=\left\langle
    \phi_{s},\psi_{\gamma }(a,b),\tau_{e},\mu _{\delta },\sigma^{\beta}:\gamma
    \in\GF{q^2}^{2}\mathit{, }s,e,\delta \in\GF{q}\mathit{, }%
    \delta \neq 0\right\rangle$$
	 by Theorem \ref{LinColl} since $\sigma^{\beta} \in \Gamma(a,b)$, $o(\sigma^{\beta})=\log_{2}q$ and $\left\langle \sigma^{\beta}\right\rangle \cap G(a,b)=1$, which is the assertion.
	\end{proof}

\section{Some orthogonal arrays}
\label{oarr}

Let $S$ be a set with $v:=|S|$ elements.
An $N\times k$ array with entries in $S$
is an \emph{orthogonal array} $OA(N,k,v,t)$
with $v$ levels, strength $t$ and index
$\lambda:=N/v^t$ if every $N\times t$ subarray of $A$ contains
each $t$-uple of elements of
$S$ exactly $\lambda$ times; see~\cite{Sloane}.
Well--known examples of orthogonal arrays are latin squares and Hadamard matrices.

There is a very rich literature about orthogonal arrays, as they
play an important role in statistics (where they are used in devising
experimental designs), cryptography (e.g. in constructing threshold schemes)
as well as in computer science (where they are used both for quality
control and for optimizing the placement and routing of elements on
PCBs). More recent applications have been found in the calibration of
the flight parameters of drones in order to optimize their
performance in the detection of some prescribed features; see~\cite{UAV}.

A general geometric
procedure for constructing an orthogonal array is as follows:
let $f_1,\ldots,f_k$ be homogeneous forms in $n+1$ unknowns defining some
algebraic varieties $V(f_1),\dots,V(f_k)$, let also
$\cW\subseteq\GF{q}^{n+1}$ be a set of representatives
of distinct points of $\Sigma=\PG(n,q)$  with $|\cW|=N$. The array
\[
	A(f_1,\ldots,f_k;\cW)= \left\{\begin{pmatrix}
		f_1(x) &
		f_2(x) &
		\dots  &
		f_k(x)
	\end{pmatrix} : x \in \cW \right\}, \]
with  an arbitrary order of rows, is orthogonal
if the size  of the intersection $V(f_i)\cap
	V(f_j)\cap\cW$ for distinct varieties $V(f_i)$ and $V(f_j)$, is
independent of the choice of $i$, $j$. This procedure was applied to
linear functions by Bose~\cite{Bo},  to quadratic functions by
Fuji-Hara and Miyamoto~\cites{FuMi1,FuMi2} and to Hermitian forms
by Aguglia and Giuzzi~\cite{A}.

In general, it is possible to generate  functions $f_i$ starting
from homogeneous polynomials in $n+1$ variables and considering the
action of a suitable subgroup of the projective group $\PGL{n+1}{q}$.
Recall that,
the image $V(f)^g$ of $V(f)$ under the action of an element $g\in
	\PGL{n+1}{q}$ is a variety $V(f^g)$ of $\Sigma$, associated with the
polynomial $f^g$.
In~\cite{FuMi2}, the authors used a
subgroup of $\PGL{4}{q}$, in order to obtain suitable quadratic
functions in $4$ variables; then, the domain $\cW$ of these
functions was appropriately restricted to a set of $q^3$
representatives, thus producing an orthogonal
array of type $OA(q^3,q^2,q,2)$.

Here, we  construct  a simple
$OA(q^{5},q^{4}, q,2)=\cA_0 $, with entries in $\GF{q}$,  $q>2$ an even prime power, using  the above procedure with
forms related to the BM quasi-Hermitian varieties $\cM_{a,b}$. To do this we look into the action of a large
subgroup of $\PGL{4}{q^2}$ on a set of BM quasi-Hermitian
varieties in $\PG(3,q^2)$.

As seen before, $\cM_{a,b}$  has the same affine points as the variety $\cB_{a,b}$  associated to the form
\[F=Z^qJ^q+ZJ^{2q-1}+a^q(X^{2q}+Y^{2q})-a(X^2+Y^2)J^{2q-2}+(b+b^q)(X^{q+1}+Y^{q+1})J^{q-1}.\]

We shall now construct an array by choosing suitable varieties of
the form $\cM_{a,b}$ lying in the orbit of a suitable set of
affine collineations.

Take
$G$ as the subgroup of $\PGL{4}{q^2}$ consisting of all elations represented by
\[c (j',x',y',z')= (j,x,y,z)M\]
where $c \in\GF{q^2}^*$, and
\begin{equation} \label{}
	M=\begin{pmatrix}
		1 & \gamma_1 & \gamma_2 & \gamma_3 \\
		0 & 1        & 0        & \gamma_4 \\
		0 & 0        & 1        & \gamma_5 \\
		0 & 0        & 0        & 1        \\
	\end{pmatrix},
\end{equation}
with $\gamma_i\in\GF{q^2}$.
Then,
the group $G$ has order
$q^{10}$, it stabilizes the hyperplane $\Sigma_\infty$, fixes
the point $P_{\infty}(0,0,0,1)$ and acts transitively on
$\AG(3,q^2)$ (i.e. it acts as an affine group of collineations).

Let now  $\Psi$ be the subgroup of $G$  consisting  of all
elations 
whose matrices are of the form
\begin{equation} \label{BM}
	\begin{pmatrix}
		1 & \gamma_1 & \gamma_2 & a(\gamma_1^2+\gamma_2^2)+ b(\gamma_1^{q+1}+\gamma_2^{q+1})+s \\
		0 & 1        & 0        & (b^q+b)\gamma_1^{q}                                         \\
		0 & 0        & 1        & (b^q+b)\gamma_2^{q}                                         \\
		0 & 0        & 0        & 1                                                                    \\
	\end{pmatrix},
\end{equation}
with $\gamma_1, \gamma_2\in\GF{q^2}$, $s\in\GF{q}$ .
The group
$\Psi$ contains  $q^{5}$ elations, preserves $\cM_{a,b}$ and acts on the
affine points of $\cM_{a,b}$, that is to say the affine points of $\cB_{a,b}$, as a sharply transitive permutation group.
Let also $C=\{a_1=0,\ldots,a_q\}$ be a set of representatives
for the elements of $\GF{q^2}/\GF{q}$ (regarding them both as
their additive groups).
Denote now by
$\cR$ the subset of $G$
whose
elations are induced by
\begin{equation} \label{collin}
	\begin{pmatrix}
		1 & \gamma_1 & \gamma_2 & \gamma_3 \\
		0 & 1        & 0        & 0        \\
		0 & 0        & 1        & 0        \\
		0 & 0        & 0        & 1        \\
	\end{pmatrix},
\end{equation}
where $\gamma_1, \gamma_2 \in\GF{q^2}$, and  $\gamma_3$ is the unique solution in
$C$ of the equation
\begin{equation}\label{ara}
	\gamma_3^q+\gamma_3+a^q(\gamma_1^{2q}+\gamma_2^{2q})+a(\gamma_1^2+\gamma_2^2)+(b+b^q)(\gamma_1^{q+1}+\gamma_2^{q+1})=0.
\end{equation}
The set  $\cR$ has cardinality $q^{4}$ and it can be seen that
its elements belong to a transversal of the group $\Psi$; in particular
$\cR$
can be used to
construct a set $\{F^g| g\in \cR\}$ of forms whose related $\cB_{a,b}$'s (and thus $\cM_{a,b}$'s) are pairwise distinct.
%
%
%
%

\begin{theorem}
	\label{teo:fh0}
	For any prime power $q$, the matrix
	$\cA_0=A(F^g,g\in \cR,\cW_0)$, where
	\[ \cW_0=\{(1,x,y,z): x,y \in  \GF{q^2}, z\in C\} \]
	is a simple $OA(q^{5},q^{4},q,2)$
	of index $\lambda=q^{3}$.
\end{theorem}
\begin{proof}
	We start by showing
	that the number of solutions in $\cW_0$ to the system
	\begin{equation}
		\left\{\begin{array}{l}
			\label{orto1}
			F(J,X,Y,Z)=t \\
			F^g(J,X,Y,Z)=t'
		\end{array}\right.
	\end{equation}
	is $q^{3}$ for any $t$, $t' \in\GF{q}$, $g\in
		\cR\setminus\{id\}$.
	By definition of  $\cW_0$, this
	system  is equivalent to
	\begin{equation}
		\left\{\begin{array}{l}
			\label{orto2}
			Z^q+Z+a^q(X^{2q}+Y^{2q})+a(X^2+Y^2)+(b+b^q)(X^{q+1}+Y^{q+1})=t \\
			Z^q+Z+a^q(X^{2q}+Y^{2q})+a(X^2+Y^2)+                                   \\
			\qquad\qquad\qquad\!\!\! (b+b^q)(X^{q+1}+Y^{q+1}+\gamma_1^qX+\gamma_2^qY+\gamma_1X^q+\gamma_2Y^q)=t'
		\end{array}\right.
	\end{equation}
	Subtracting the first equation from the second we get
	\begin{equation}\label{trc:2}
		\tr(\gamma_1^qX+\gamma_2^qY)=\frac{t+t'}{(b+b^q)},
	\end{equation}
	Since  $g$ is not the identity,
	$(\gamma_1,\gamma_2) \neq (0,0)$;  hence,
	Equation~\eqref{trc:2} is equivalent to the union of $q$ linear equations
	in $X,Y$ over $\GF{q^2}$. Thus, there are
	$q^{3}$ pairs $(x,y)$ satisfying~\eqref{trc:2}.
	For each such a pair, Equation~\eqref{orto2} has $q$ solutions in $Z$,
	corresponding to a coset  of $\GF{q}$ in $\GF{q^2}$ and only one of these $q$ solutions is in $C$. Therefore,
	System~\eqref{orto1} has $q^3$ solutions in $\cW_0$.
	
	Next, we show that $\cA_0$ does not contain any
	repeated row.
	Let us index its rows by the
	corresponding elements in $\cW_0$.
	Observe that the
	row
	$(x,y,z)$
	is the same as
	$(x_1,y_1,z_1)$
	in $\cA_0$
	if, and only if,
	\[ F^g(1,x,y,z)=F^g(1,x_1,y_1,z_1), \]
	for any $g\in \cR$.
	We thus obtain a system of $q^{4}$
	equations in the $6$ indeterminates $x_1,y_1,z_1,x,y,z$. Each equation is of the form
	\begin{equation}
		\begin{array}{l}
			\label{eq:scol}
			(z+z_1)^q+(z+z_1)+a^q((x+x_1)^{2q}+(y+y_1)^{2q}) \\
			+a((x+x_1)^2+(y+y_1)^2)+
			(b+b^q)((x+x_1)^{q+1}+(y+y_1)^{q+1})=         \\
			(b^q+b)(\gamma_2^q(y+y_1)+\gamma_2(y+y_1)^q+\gamma_1^q(x+x_1)+\gamma_1(x+x_1)^q)
		\end{array}
	\end{equation}
	where the elements $\gamma_i$ vary in $\GF{q^2}$ in
	all possible ways.
	In particular,
	for $\gamma_i=0$ we have that
	the left hand side of the equations of~\eqref{eq:scol} equals zero.
	Thus,
	\begin{equation}
		\label{eq:scol1}
		(b^q+b)(\gamma_2^q(y+y_1)+\gamma_2(y+y_1)^q+\gamma_1^q(x+x_1)+\gamma_1(x+x_1)^q) =0
	\end{equation}
	Choosing $\gamma_1=1$ and $\gamma_2=0$ it follows from~\eqref{eq:scol1} that $x$ and $x_1$ must be in the same coset of
	$\GF{q}$. If we choose $\gamma_1=0$ and $\gamma_2=1$ in~\eqref{eq:scol1} we get that $y$ and $y_1$ are as well in the same coset of
	$\GF{q}$; so $x+x_1,y+y_1\in\GF{q}$.
        So, \eqref{eq:scol} becomes
        \begin{equation}\label{cos1}
        \tr(z+z_1)+\tr(a+b)\left((x+x_1)+(y+y_1)\right)^2=0 
        \end{equation}
        and  \eqref{eq:scol1} becomes
          \begin{equation}\label{cos2}
        \tr(b)\left(\tr(\gamma_2)(y+y_1)+\tr(\gamma_1)(x+x_1)\right)=0.
        \end{equation}
        By assumption $\tr(b)\neq0$ and by the arbitrariness of $\gamma_1$ and $\gamma_2$     \eqref{cos2} gives $x=x_1$ as well as $y=y_1$.

    Then,  \eqref{cos1} implies $\tr(z+z_1)=0$
        that is, $z$ and $z_1$ are in the same coset of $\GF{q}$.
  Thus, there are
	no two distinct vectors in $\cW_0$ whose difference
	is of the required form; hence, $\cA_0$ does
	not contain repeated rows
	and the theorem follows.
      \end{proof}

\section*{Acknowledgments}
All of the authors thank the Italian National Group for Algebraic and Geometric Structures and their Applications (GNSAGA--INdAM) for its support to their research.
The work of A. Aguglia and V. Siconolfi has also been
 partially founded by the European Union under the Italian National Recovery and Resilience Plan (NRRP) of NextGenerationEU, partnership on “Telecommunications of the Future” (PE00000001 - program ‘‘RESTART’’, CUP: D93C22000910001) and by the Italian Ministry of University and Research under the Programme “Department of Excellence” Legge 232/2016 (Grant No. CUP - D93C23000100001).


\end{document}